\documentclass[a4paper,11pt]{amsart}

\usepackage[T1]{fontenc}
\usepackage[english]{babel}
\usepackage[utf8]{inputenc}
\usepackage{lmodern}
\usepackage{amsmath}
\usepackage{amsthm}
\usepackage{amssymb}
\usepackage{enumerate}
\usepackage{multienum}
\usepackage{amsfonts}
\usepackage{hyperref}
\usepackage{multicol}
\usepackage{enumitem}
\usepackage{graphicx}
\usepackage{array}
\usepackage{caption}
\usepackage{float}
\usepackage{tikz}
\usepackage{amsaddr}

\usepackage[top=1in, bottom=1in, left=1.25in, right=1.25in]{geometry}

\newtheorem{thm}{Theorem}[section]
\newtheorem{prop}[thm]{Proposition}
\newtheorem{lem}[thm]{Lemma}
\newtheorem{cor}[thm]{Corollary}
\newtheorem{con}{Conjecture}

\theoremstyle{definition}

\newtheorem{rem}{Remark}
\newtheorem{prob}{Problem}
\newtheorem{ex}{Example}

\newcommand{\R}{\mathbb{R}}

\newcommand{\N}{\mathbb{N}}

\newcommand{\Pref}{\operatorname{Pref}}
\newcommand{\Suf}{\operatorname{Suf}}
\newcommand{\Sub}{\operatorname{Sub}}

\begin{document}
\title[On monochromatic arithmetic progressions in binary words]{On monochromatic arithmetic progressions in binary words associated with pattern sequences}

\author{Bartosz Sobolewski}
\address{Institute of Mathematics, Faculty of Mathematics and Computer Science, Jagiellonian University in Krak\'{o}w, \\
 \L{}ojasiewicza 6, 30-348, Krak\'{o}w, Poland}
\email{bartosz.sobolewski@uj.edu.pl}

\begin{abstract}
Let $e_v(n)$ denote the number of occurrences of a fixed pattern $v$ in the binary expansion of $n \in \N$. In this paper we study monochromatic arithmetic progressions in the class of binary words $(e_v(n) \bmod{2})_{n \geq 0}$, which includes the famous Thue--Morse word $\mathbf{t}$ and Rudin--Shapiro word $\mathbf{r}$. We prove that the length of a monochromatic arithmetic progression of difference $d \geq 3$ starting at $0$ in $\mathbf{r}$ is at most $(d+3)/2$, with equality for infinitely many $d$. Moreover, we compute the maximal length of a monochromatic arithmetic progression in $\mathbf{r}$ of difference $2^k-1$ and $2^k+1$. For a general pattern $v$ we provide an upper bound on the length of a monochromatic arithmetic progression of any difference $d$. We also prove other miscellaneous results and offer a number of related problems and conjectures.
\end{abstract}

\maketitle

\section{Introduction} \label{sec:introduction}

Let $\mathbf{t} = (t_n)_{n \geq 0}$ be the Thue--Morse sequence, given by $t_n = s_2(n) \bmod{2}$, where $s_2$ denotes the sum of binary digits. There is a vast literature devoted to the properties of $\mathbf{t}$. For a survey concerning its various appearances in seemingly unrelated problems see Allouche and Shallit \cite{AS99}. A comprehensive treatment in the setting of automatic sequences is given in the monograph by the same authors \cite{AS03}.

The sequence $\mathbf{t}$ can be identified with an infinite binary word or a $2$-coloring of the set $\N$ of nonnegative integers, namely a partition into two sets of indices $n$ such that $t_n=0$ and $t_n =1$, respectively. By a celebrated result of van der Waerden \cite{vdW27}, the coloring $\mathbf{t}$ contains arbitrarily long finite monochromatic arithmetic progressions, that is, progressions $n,n+d,\ldots,n+ld $ such that $t_n = t_{n+d} =  \cdots = t_{n+ld}$.
In fact, Avgustinovich, Cassaigne and Frid \cite{ACF06} proved a stronger property of $\mathbf{t}$, namely that all finite binary words occur in $\mathbf{t}$ along arithmetic progressions.

These results, however, do not tell how long a monochromatic arithmetic progression of fixed difference $d$ can be. Morgenbesser, Shallit, and Stoll \cite{MSS11} addressed this problem for progressions starting at $0$ and gave a sharp upper bound on their maximal length. Here and in the sequel we let $A_{\mathbf{t}}(n,d)$ denote the maximal length of a monochromatic arithmetic progression in $\mathbf{t}$ of difference $d$ starting at $n$, that is,
$$ A_{\mathbf{t}}(n,d) = \inf\{l \geq 1: t_{n+ld} \neq t_n \}.    $$
The main result of \cite{MSS11} is the following. (In the original statement the notation $f(d)$ was used instead of $A_{\mathbf{t}}(0,d)$.)

\begin{thm}[\cite{MSS11}] \label{thm:MSS11}
For all $d \in \N_{+}$ we have
$$A_{\mathbf{t}}(0,d) \leq d+4.$$
Moreover,
\begin{enumerate}[label={\textup{(\roman*)}}]
\item $A_{\mathbf{t}}(0,d) = d+4$ if and only if $d = 2^{k}-1$ for some $k \geq 2$ even;
\item there exist no $d$ such that $A_{\mathbf{t}}(0,d)=d+3$ or $A_{\mathbf{t}}(0,d) = d+2$;
\item $A_{\mathbf{t}}(0,d) = d+1$ if and only if $d=6$;
\item $A_{\mathbf{t}}(0,d) = d$ if and only of $d=1$ or $d=2^k+1$ for some $k \geq 2$.
\end{enumerate}
Furthermore, for any $d \in \N_{+}$ there exists $n \leq d+4$ such that $t_{dn}=1$ and $s_2(n)=3$.
\end{thm}
In the same paper it is also observed that $\mathbf{t}$ does not contain any infinite monochromatic arithmetic progressions (as a consequence of a result of Gelfond \cite{Gel67}).

Another theorem on monochromatic arithmetic progressions in $\mathbf{t}$ and its generalizations to other bases was obtained by Parshina \cite{Par17}. She showed that the lengths of all monochromatic progressions of difference $d$ are bounded by a function of the length of binary expansion $d$. More precisely, define
$$ A_{\mathbf{t}}(d) = \sup_{n \in \N} A_{\mathbf{t}}(n,d).   $$
In the case of the Thue--Morse sequence, Parshina's result is the following.

\begin{thm}[\cite{Par17}] \label{thm:Par17}
For all integers $k \geq 1$ we have
$$  \max_{1 \leq d \leq 2^k-2} A_{\mathbf{t}}(d)  \leq  2^k$$
and
$$   A_{\mathbf{t}}(2^k-1) =  \begin{cases}
2^k + 4, &\text{if } k \equiv 0 \pmod{2},\\
2^k &\text{otherwise}.
\end{cases} $$
\end{thm}

We remark the second part of the statement is obtained by incorporating \cite[Lemma 1]{Par17}.
In comparison to Theorem \ref{thm:MSS11}, the assumption about the progression starting at $n = 0$ is relaxed, at the cost of losing a more precise bound for each individual $d$.

The expression for $A_{\mathbf{t}}(2^k-1)$ was proved in another way by Aedo, Grimm, Nagai, Staynova \cite{AGNS22} together with the following theorem (and other results concerning a class of Thue--Morse-like sequences).

\begin{thm}[\cite{AGNS22}] \label{thm:AGNS22}
For all integers $k \geq 2$ we have $A_{\mathbf{t}}(2^k+1) = 2^k+2$.
\end{thm}

Motivated by these results, in this paper we consider another direction of generalization than those in \cite{Par17,AGNS22}, obtained by replacing the sum of binary digits $s_2$ by other pattern-counting functions. More precisely, let $v$ be a nonempty finite word consisting of the digits $0,1$. Define $e_v(n)$ to be the number of (possibly overlapping) occurrences of $v$ in the binary expansion of $n \in \N$. In this context we call $v$ a \emph{pattern} and $(e_v(n))_{n \geq 0}$ -- a \emph{pattern sequence}. Following \cite[p.\ 78]{AS03}, we adhere to the convention that if $v$ begins with precisely $i \geq 1$ zeros and but at least one nonzero digit, then in the computation of $e_v(n)$ the binary expansion of $n$ is preceded with $i$ zeros. For example, if $v=001$ and $n=9$, then we write the binary expansion of $9$ as $001001$, which gives $e_{001}(9) = 2$.
 
In particular, we have $e_1(n) = s_2(n)$ so $(e_1(n) \bmod 2)_{n \geq 0}$ is precisely the Thue--Morse sequence. We are going to focus most of our attention on a variant of the famous Rudin--Shapiro sequence $\mathbf{r} = (r_n)_{n \geq 0} = (e_{11}(n) \bmod{2})_{n \geq 0}$. This sequence is usually defined over $\{-1,1\}$ by $r_n = (-1)^{e_{11}(n)}$. We refer the reader to \cite{AS03} for an exposition regarding the Rudin--Shapiro sequence.

Similarly as before, we can treat the sequence $(e_v(n) \bmod{2})_{n \geq 0}$ as an infinite word over $\{0,1\}$ or a $2$-coloring of nonnegative integers. For $n,d \in \N$ such that $d \geq 1$, we define
\begin{align*}
 A_v(n,d) &= \inf\{l \geq 1: e_v(n + ld) \not\equiv e_v(n) \pmod{2} \}, \\
 A_v(d) &= \sup_{n \in \N} A_v(n,d).
\end{align*}
In other words, $A_v(n,d)$ is the maximal (possibly infinite) length of a monochromatic arithmetic progression of difference $d$ and starting at $n$ in said coloring. The value $A_v(d)$, if finite, is the maximal length of such a progression starting anywhere. We note that $A_v(d)$ may potentially be infinite even if all $A_v(n,d)$ are finite.
In the case of the Rudin--Shapiro sequence, namely $v=11$, we adopt the special notation $A_{\mathbf{r}}(n,d) = A_{11}(n,d)$ and $A_{\mathbf{r}}(d) = A_{11}(d)$. 

We now present the main results of this paper, which are proved in the later sections.
The first result is an analogue of Theorem \ref{thm:MSS11} for the Rudin--Shapiro sequence.

\begin{thm} \label{thm:main1}
For all integers $d \geq 3$ we have
$$  A_{\mathbf{r}}(0,d) \leq \frac{1}{2}(d+3).   $$
More precisely:
\begin{enumerate}[label={\textup{(\roman*)}}]
\item $A_{\mathbf{r}}(0,d) = (d+3)/2$ if and only if $d = 2^k-1$ for some odd $k \geq 3$ or $d=39$;
\item $A_{\mathbf{r}}(0,d) = (d+1)/2$ if and only if $d = 2^k+1$ for some $k \geq 2$;
\item $A_{\mathbf{r}}(0,d) < d/2$ otherwise.
\end{enumerate}
\end{thm}
 
The next two results give an explicit expression for $A_{\mathbf{r}}(d)$ for the special cases $d = 2^k+1$ and $d = 2^k-1$.

\begin{thm} \label{thm:main2}
For all integers $k \geq 1$ we have
$$
A_{\mathbf{r}}(2^k+1) = \begin{cases}
5 &\text{if } k=1, \\
6 &\text{if } k=2, \\
9 &\text{if } k=3, \\
2^{k-1} + 2 &\text{if } k \geq 4.
\end{cases}
$$ 
\end{thm}

\begin{thm} \label{thm:main3}
For all integers $k \geq 1$ we have
$$
A_{\mathbf{r}}(2^k-1) = \begin{cases}
4 &\text{if } k=1, \\
5 &\text{if } k=2, \\
9 &\text{if } k=3, \\
10 &\text{if } k=4, \\
2^{k-1} + 3 &\text{if } k \geq 5 \text{ is odd}, \\
2^{k-1} + 1 &\text{if } k \geq 6 \text{ is even}.
\end{cases}
$$ 
\end{thm}

\begin{rem}
In a recent preprint, Aedo et al.\ \cite{AGMNP23} considered asymptotic growth of lengths of monochromatic arithmetic
progressions in a certain family of infinite words, which includes the Rudin--Shapiro word. In particular, they obtained  lower bounds for $A_{\mathbf{r}}(2^k \pm 1)$ using a different method.
\end{rem}

These results do not yet settle whether or not the values $A_{\mathbf{r}}(n,d)$ and $A_{\mathbf{r}}(d)$ are finite for general $n,d$. The following result shows that this is in fact the case for any pattern $v$ of length at least $2$, and provides an upper bound in the spirit of Theorem \ref{thm:Par17}. In its statement $\nu_2(d)$ denotes the $2$-adic valuation of $d$, namely $\nu_2(d) = \sup \{n \in \N: 2^n \mid d\}$.

\begin{thm} \label{thm:general_bound}
For any pattern $v$ of binary digits of length $|v| \geq 2$ and all integers $k \geq 1$ we have 
$$\max_{1 \leq d < 2^k} A_v(d) \leq 2^{k+|v|-\nu_2(d)-1}.$$
\end{thm}

\begin{rem}
Theorem  \ref{thm:general_bound} also holds in the Thue--Morse case $v=1$ for $k$ odd but fails for $k$ even. This follows immediately from Theorem \ref{thm:Par17} and the equality $A_{\mathbf{t}}(2d) = A_{\mathbf{t}}(d)$ for all $d \geq 1$, which was stated in \cite[p.\ 5]{AGNS22}. The same conclusion also holds for the pattern $v=0$ due to the equality $A_{0}(d) = A_{\mathbf{t}}(d)$ (see Proposition \ref{prop:A_0=A_t} below).
\end{rem}

We now give an outline of the contents of this paper. In Section \ref{sec:notation} we give some general notation and terminology used throughout the paper. Section \ref{sec:main1} is devoted to the proof of Theorem \ref{thm:main1}. Parts (i) and (ii) are proved using straightforward computation, while part (iii) relies on detailed case distinction performed with the help of computer software. The goal of Section \ref{sec:main2_main3} is to prove Theorems \ref{thm:main2} and \ref{thm:main3}. To achieve this, we use a nice relation linking $\mathbf{r}$ and its subsequence $(r_{2n+1})_{n \geq 0}$ (Proposition \ref{prop:interesting_property}) and some standard properties of the subwords of $\mathbf{r}$. Section \ref{sec:general_v} contains the proof of Theorem \ref{thm:general_bound} together with other general results on the behavior of $A_v(0,d)$ and $A_v(d)$. In particular, we provide an infinite family of sequences $(d_k)_{k \geq 0}$ such that $A_v(0,d_k)$ is ``almost'' linear in $d_k$. In Section \ref{sec:problems} we discuss likely improvements to the results presented in this paper and state many related problems and conjectures, backed by experimental calculations.  The method used to compute the values $A_v(d)$ is described in Section \ref{sec:calculation}. Finally,  Section \ref{sec:supplemental} contains a description of the supplemental files, which are available in the repository \cite{Git}.

\section{Notation and terminology} \label{sec:notation}

We introduce some further notation and terminology related to operations on words, mainly following \cite[Chapter 1]{AS03}. Let $\Sigma$ be a nonempty finite set (an alphabet). Then $\Sigma^*$ denotes the set of all finite words (blocks, strings) consisting of the letters from $\Sigma$, including the empty word $\epsilon$. We also write $\Sigma^{+} = \Sigma^* \setminus \{\epsilon\}$. Two words $v,w \in \Sigma^*$ can be concatenated into a single word $vw$. For $n \in \N$ the notation $w^n$ means the concatenation of $n$ copies of $w$, where $w^0 = \epsilon$. The length $|w|$ of a word is the number of its letters, and for each $l \in \N$ we let $\Sigma^l$ denote the set of words of length $l$. The reversal of a word $w = w_1 w_2 \cdots w_l$ is given by $w^R = w_l w_{l-1} \cdots w_1$. We say that $v$ is a subword of $w$ if there exist words $x,y$ such that $w = xvy$ and we let $\Sub(w)$ denote the set of subwords of $w$. If $w=vy$, then $v$ is said to be a prefix of $w$. If additionally $y \neq \epsilon$, then the prefix $v$ is called proper. Similarly, if $w = xv$, then $v$ is said to be a suffix of $w$ (proper when $x \neq \epsilon$). The sets of prefixes and suffixes of $w$ are denoted by $\Pref(w)$ and $\Suf(w)$, respectively. The number of (possibly overlapping) occurrences of a nonempty word $v$ as a subword in $w$ is denoted by $|w|_v$. The notions of subwords and prefixes naturally extend to the case where $w=w_0w_1\cdots$ is a right-infinite word.

We now focus on the binary case $\Sigma = \{0,1\}$.
For any $w = w_k w_{k-1} \cdots w_1 w_0 \in \{0,1\}^+$ the notation $[w]_2$ refers to the integer represented by $w$ in base $2$, namely
$$ [w]_2 = 2^k  w_k + 2^{k-1} +  \cdots + 2 w_1 + w_0. $$
We also put $[\epsilon]_2 = 0$. Note that adding leading zeros to $w$ does not affect $[w]_2$. Conversely, for $n \in \N$ we let $(n)_2$ denote the canonical binary representation of $n$, that is, the unique word $w \in \{0,1\}^*$ without leading zeros such that $[w]_2 = n$. The length of the binary expansion of $n$ is denoted by $\ell(n)$, and can be equivalently written as
$$ \ell(n) = |(n)_2| = 1 + \lfloor \log_2 n \rfloor. $$
In accordance with our earlier convention, when $v$ is a word starting with exactly $i \geq 0$ zeros and containing at least one $1$, we have
$$  e_{v}(n) = |0^i (n)_2|_v.  $$
Finally, we let $\overline{w}$ denote the binary negation of a word $w \in \{0,1\}^*$, obtained by swapping $1$'s with $0$'s and vice versa.

\section{Proof of Theorem \ref{thm:main1}} \label{sec:main1}

\subsection{Parts (i) and (ii)}

To begin, we deal with parts (i) and (ii) in the statement of Theorem \ref{thm:main1}. In the case (i) we prove a more precise result, which will be useful later when calculating $A_{\mathbf{r}}(d)$.

\begin{prop} \label{prop:special_case_1}
If $d = 2^{k}-1$ for some integer $k \geq 3$, then
$$ r_d =  r_{2d} = \cdots = r_{2^{k-1} d},$$
and this common value is equal to  $(k-1) \bmod{2}$.
Moreover, if $k$ is odd, then $r_{(2^{k-1}+1) d} = k \bmod{2}$, and consequently $A_{\mathbf{r}}(0,d) = (d +3)/2$.
\end{prop} 
\begin{proof}
Since $e_{11}(2m) = e_{11}(m)$ for all $m \in \N$, it is sufficient to show that $e_{11}(nd) \equiv k-1 \pmod{2}$ for all odd $n < 2^{k-1}$. 
Fix any such $n$ and let $w \in \{0,1\}^{k-2}$ be such that $[w0]_2 = n-1$ ($w$ may have leading zeros). Then
$$nd = 2^{k}(n-1) + (2^{k}-1)-(n-1) = [0w01\overline{w}1]_2.$$
Therefore,
$$e_{11}(nd) = |0w0|_{11} + |1\overline{w}1|_{11} = |0w0|_{11} + |0w0|_{00}.$$
Now, for any $u \in \{0,1\}^+$ we have
$$|u|_{11} + |u|_{00} + |u|_{01} + |u|_{10} = |u|-1,$$
since both sides count the number of blocks of adjacent digits of length $2$ in $u$.
At the same time, observe that
$$ |0w0|_{01} = |0w0|_{10}.$$
Combining the above equalities, we obtain
$$ e_{11}(nd) = |0w0| - 1 - |0w0|_{01} - |0w0|_{10} = k - 1 - 2|0w0|_{01} \equiv k-1 \pmod{2},$$
as desired.

If $k$ is odd, then the values $r_{nd}$ for $n = 0,1,\ldots,2^{k-1}$ are all equal to $0$. However, for $n = (d + 3)/2 = 2^{k-1} + 1$ we have $nd = 2^{2k-1} + 2^{k-1} - 1 = [10^k 1^{k-1}]_2$, so $e_{11}(nd) = k-2$ and the second part of the assertion follows.
\end{proof}

Direct calculation shows that we have $A_{\mathbf{r}}(0,d) = (d+3)/2$ also for $d=39$.

We move on to part (ii) of Theorem \ref{thm:main1}, which is quite simple to prove.

\begin{prop} \label{prop:special_case_2}
If $d = 2^{k}+1$ for some integer $k \geq 2$, then $A_{\mathbf{r}}(0,d) = (d + 1)/2.$
\end{prop} 
\begin{proof}
First, we prove that $r_{nd} = 0$ for all $n < (d + 1)/2 = 2^{k-1}+1$. Again, it is sufficient to consider $n$ odd. We have $(nd)_2 = (n)_2 0^m (n)_2$, where $m = k - \ell(n) > 0$. This means that
$$  e_{11}(nd) = 2e_{11}(n) \equiv 0 \pmod{2}. $$
Now, for $n = 2^{k-1}+1$ we have $(nd)_2 = 1 0^{k-2} 1 1 0^{k-2} 1$, which gives $e_{11} (nd) = 1$. The result follows.
\end{proof}

\subsection{Part (iii) --- general approach} 

We move on to part (iii) of Theorem \ref{thm:main1}, which requires much more effort to prove. The relation $r_{2m} = r_{m}$ implies $A_{\mathbf{r}}(0,2d) = A_{\mathbf{r}}(0,d)$, and thus allows us to restrict our attention to odd differences $d$. We can also assume that there is a $0$ in the binary expansion of $d$. Indeed, if $(d)_2 = 1^k$ for some $k \geq 2$, then for $k$ even we get $A_{\mathbf{r}}(0,d)=1 < d/2$, while $k$ odd falls under part (i) of the theorem. 

We now give an outline of our method. The goal is to show that for each $d$ covered by part (iii) there exists a positive integer $n < d/2$ such that $r_{nd} = 1$. To achieve this we adapt the approach used in \cite{MSS11} to prove Theorem \ref{thm:MSS11}. More precisely, instead of trying to find a single specific value $n$, we look for two positive integers $m,n$ such that $m+n < d/2$ and at least one of $r_{nd}, r_{md}, r_{(m+n)d}$ is equal to $1$. The key here is to relate binary expansions of $nd, md$ and $(m+n)d$. In order to do this we need some prior knowledge about the binary expansion of $d$, hence the analysis is split into multiple cases, depending on a prefix and suffix of the binary expansion of $d$. More precisely let $p,s \in \{0,1\}^+$ be such that $p$ starts with a $1$ and $s$ ends with a $1$. We define the set $D(p,s)$ containing $d$ such that $(d)_2$ starts with $p$ and ends with $s$, namely
$$  D(p,s) = \{d \in \N: \ p \in \Pref((d)_2), s \in \Suf((d)_2)\}.  $$
Before giving a more detailed explanation, we give an example illustrating the basic concept. Here and in the sequel we will write $\ell = \ell(d)$.
\begin{ex} \label{ex:basic_concept}
We will show that $A_{\mathbf{r}}(0,d) < d/2$ for all $d \in D(11011,11101)$. We claim that at least one of $r_{d}, r_{5d}, r_{(2^{\ell-4} \cdot 5+1)d}$ is equal to $1$. The suffix of $(5d)_2$ of length $5$ is $10001$ and we can write $(5d)_2 = w10001$ for some $w \in \{0,1\}^*$. Similarly, we can write $(d)_2=11011u$ for some $u \in \{0,1\}^*$. We now perform binary addition $(d)_2+ (2^{\ell-4} \cdot 5d)_2= (2^{\ell-4} \cdot 5+1)d$, obtaining:
$$\begin{array}{lllllllllll}
       &   &   & 1 & 1 & 0 & 1 & 1 & u \\
+      & w & 1 & 0 & 0 & 0 & 1 &   &   \\
\hline & w & 1 & 1 & 1 & 1 & 0 & 1 & u
\end{array}.$$
If we count the number of occurrences of $11$ in each line and add them together, we get
$$ (2 + |1u|_{11}) + |w1|_{11} + (|w1|_{11} + 3 + |1u|_{11}) = 5 + 2(|1u|_{11} + |w1|_{11}) \equiv 1  \pmod{2}. $$
Hence, at least one line contains an odd number of occurrences of $11$ and our claim holds. Since $(2^{\ell-4} \cdot 5+1)d$ is the largest of the considered indices, we get
$$ A_{\mathbf{r}}(0,d) < 2^{\ell-4} \cdot 5+1.$$
At the same time, we have $\ell(2^{\ell-4} \cdot 5+1) = \ell-1$ and the prefix $10$ of $(2^{\ell-4} \cdot 5+1)_2$ is lexicographically smaller than the prefix $11$ of $(d)_2$. These two conditions together give $2^{\ell-4} \cdot 5+1 < d/2$, which is what we needed to show. 
We point out that the whole argument also holds for the elements of $D(p,s)$ in whose binomial expansion the prefix $p$ and suffix $s$ overlap: $[11011101]_2$ and $[110111101]_2$. 
\end{ex}

Ideally, we would like to partition the set of odd positive integers (up to some exceptional cases) into a family of sets $D(p,s)$, each of which can be treated as in Example \ref{ex:basic_concept}. To make this more precise, let $i,j$ be positive integers, where we assume that $j$ is odd and $i<|s|$. We will say that the pair $(i,j)$ is \emph{admissible} for $(p,s)$ if the following two conditions are satisfied for all $d \in D(p,s)$:
\begin{align}
  2^{\ell-i} j+1 &< \frac{1}{2} d,  \tag{A1} \label{eq:A1} \\
  e_{11}(d) + e_{11}(jd) + e_{11}((2^{\ell-i} j+1)d) &\equiv 1 \pmod{2}. \tag{A2} \label{eq:A2}
\end{align}
Example \ref{ex:basic_concept} shows that $(i,j) = (4,5)$ is an admissible pair for $(p,s)=(11011,11101)$. Generalizing Example, we can state the following key observation.

\begin{prop} \label{prop:admissible}
Let $p,s \in \{0,1\}^*$ be such that $p$ begins with a $1$ and $s$ ends with a $1$. If there exists an admissible pair for $(p,s)$, then for all $d \in D(p,s)$ we have $A_{\mathbf{r}}(0,d)<d/2$. 
\end{prop}

We now discuss how to verify admissibility. Condition \eqref{eq:A1} is rather straightforward to check by comparing binary expansions of $d$ and $(2^{\ell-i} j+1)$.
Analyzing Example \ref{ex:basic_concept} more closely, we can observe that the following conditions were sufficient for \eqref{eq:A2}:
\begin{itemize}    
    \item the total number of occurrences of $11$ between $w$ and $u$ is odd;
    \item the digits adjacent to $w$ and $u$ do not change.
\end{itemize}

For further reference, we state a these observations formally as a lemma.

\begin{lem} \label{lem:main_idea}
Let $p,s \in \{0,1\}^*$ be such that $p$ begins with a $1$ and $s$ ends with a $1$, and $|s| \geq |p|$. Let $i,j$ be positive integers, where $j$ is odd and $i < |p|$. Furthermore, let $\sigma$ denote the suffix of $(j[s]_2)_2$ of length $|s|$, and let $\gamma$ denote the result of the binary addition $p + \sigma 0^{|p|-i}$, where the initial zeros of $\sigma$ (if any) are taken into account. Assume that the following conditions hold:
\begin{enumerate}[label={\textup{(\alph*)}}]
\item either $\ell(j) < i-1$, \\
or $\ell(j) = i-1$ and $2^{-\ell(j)} j < 2^{-|p|} [p]_2$;
\item $|p|_{11}+ |\sigma|_{11}+|\gamma|_{11}$ is odd;
\item there is no carry added to the leftmost digit of $\sigma$ in the addition $p + \sigma 0^{|p|-i}$.
\end{enumerate}
Then $(i,j)$ is admissible for $(p,s)$.
\end{lem}
\begin{proof}
We first argue that (a) implies \ref{eq:A1}. If $\ell(j) = i-1$, then $\ell(2^{\ell-i}j+1) < \ell-1$ so \eqref{eq:A1} holds. Otherwise, if $\ell(j) = i-1$, then the expansions $(2(2^{\ell-i}j+1))_2$ and $(d)_2$ have equal length, and $2^{-\ell(j)} j < 2^{-|p|} [p]_2$ says that the former expansion has a lexicographically smaller prefix of length $|p|$, which again gives \eqref{eq:A1}.

We now show that (b) and (c) together imply \eqref{eq:A2}.
For $d \in D(p,s)$ write $(d)_2 = pu$ and $(jd)_2 = w \sigma$ for some $u,w \in \{0,1\}^*$. 
Let $p',\sigma' \in \{0,1\}^{|p|+|\sigma|-i-2}$ and $a,b \in \{0,1\}$ be such that $0^{|\sigma|-i-1}p=p'b$ and $\sigma 0^{|p|-i-1}=a\sigma'$. If we let $\gamma'$ denote the result of binary addition $p'+\sigma'$, then $\gamma=a\gamma'b$ by the assumption (c). Therefore, binary addition $ pu + w\sigma 0^{|p|+|u|-i} = (d)_2+ (2^{\ell-i} jd)_2$ can be written as
$$
\begin{array}{llllll}
		&   &   & p'  	      &  b & u  \\
+    	& w & a & \sigma'     &    &  	\\   	
\hline  & w & a & \gamma'  	  &  b & u
\end{array}.
$$
Counting the occurrences of the pattern $11$ in each line, we obtain
\begin{align*}
e_{11}(kd) &= |p'b|_{11}+|cu|_{11} = |p|_{11} + |bu|_{11}, \\
e_{11}(jd) &= |wa|_{11}+|a \sigma'|_{11} = |wa|_{11} + |\sigma|_{11},\\
e_{11}((2^{\ell-i} j+1)d) &= |wa|_{11} + |a\gamma'b|_{11} + |bu|_{11} = |wa|_{11} + |\gamma|_{11} + |bu|_{11}.
\end{align*}
Subtracting the first and second equalities from the third one, we get
$$ e_{11}((2^{\ell-i}j+1)d) - e_{11}(kd) - e_{11}(jd) = |\gamma|_{11} - |\pi|_{11} - |\sigma|_{11},   $$
and the result follows from assumption (b).
\end{proof}

\begin{rem}
    We point out that in order for (a) to be satisfied it is necessary that $|s| \geq |p| > i \geq 2$.
\end{rem}

\subsection{Finding admissible pairs --- the algorithm}

Note that for given $(p,s)$ there exist only finitely many pairs $(i,j)$ satisfying  assumption (a) of Lemma \ref{lem:main_idea}. Hence, we can successively apply the lemma to each such pair $(i,j)$ in hope of finding an admissible one. If this fails, the reasoning can be split into multiple cases by specifying more digits in the prefix and suffix. However, doing this by hand is extremely tedious and quickly becomes intractable.

Instead, we use computer software to carry out this procedure. We now describe a simple iterative algorithm whose goal is to prove that $A_{\mathbf{r}}(0,d) < d/2$ for possibly many $d \in D(p,s)$ and identify exceptional cases.

The algorithm takes as input two blocks $p_1,s_1 \in  \{0,1\}^+$ satisfying the conditions:
\begin{itemize}
\item $p_1$ begins with a $1$;
\item $s_1$ ends with a $1$;
\item $|s_1| \geq |p_1|$.
\end{itemize}
We also define the following notation:
\begin{itemize}
\item $m \geq 1$ -- the iteration number;
\item $K_m$ -- the set of consisting of those $(p,s)$ for which we seek admissible pairs in the $m$th iteration, where $K_1 = \{(p_1,s_1)\}$;
\item $S_m$ -- the set of $(p,s) \in K_m$ for which an admissible pair has been found in the $m$th iteration;
\item $E_m$ -- a (finite) set of exceptional values $d$ for which $A_{\mathbf{r}}(d) \geq d/2$, found in the $m$th iteration.
\end{itemize}
In the $m$th iteration the algorithm performs for each $(p,s) \in K_m$ the following steps.
\begin{enumerate}
\item[\textbf{Step 1.}] If $p \in \Pref(s)$ and $A_{\mathbf{r}}(0,[s]_2) \geq [s]_2/2$, then add $d = [s]_2$ to $E_m$.
\item[\textbf{Step 2.}] For all pairs positive integers $(i,j)$ such that $j$ is odd, $2 \leq i<|p|$, and satisfying assumption (a) of Lemma \ref{lem:main_idea}, check whether conditions (b), (c) of the lemma are satisfied. If yes for some $(i,j)$, add $(p,s)$ to $S_m$ and skip to Step $1$ for the next pair $(p,s) \in K_m$
\item[\textbf{Step 3.}] Since no admissible pair $(i,j)$ has been found, add to $K_{m+1}$ the pairs $(p0,0s)$, $(p0,1s)$, $(p1,0s)$, $(p1,1s)$ and move to Step 1 for the next pair $(p,s) \in K_m$.
\end{enumerate}

The goal is to obtain $K_{m+1} = \varnothing$ at the end of the $m$th iteration, in which case the algorithm terminates and returns the sets 
$$ S^{(m)} = \bigcup_{j=1}^m S_j, \qquad E^{(m)} = \bigcup_{j=1}^m E_j.$$
Since the algorithm might run indefinitely, we may want to force it to terminate after at most $m$ iterations. In this case the algorithm also outputs $K_m \setminus S_m$ -- the set of pairs $(p,s)$ for which no admissible $(i,j)$ has been found in the $m$th iteration.

The reason behind considering the sets $E_j$ is that splitting the case $d \in D(p,s)$ into ``branches'' $(p0,0s)$, $(p0,1s)$, $(p1,0s)$, $(p1,1s)$ omits the element $d = [s]_2 \in D(p,s)$ if $p \in \Pref(s)$. Hence, Step 1 is a precaution against this possibility. 

As a consequence of Lemma \ref{lem:main_idea} and the above discussion, we get the following result.

\begin{prop} \label{prop:algorithm_justification}
Let $(p_1,s_1)$ be the input of the algorithm and let $m \geq 1$. Then for all 
$$  d \in \left(\bigcup_{(p,s) \in S^{(m)}} D(p,s) \right) \setminus E^{(m)}$$
we have $A_{\mathbf{r}}(0,d)  < \frac{1}{2} d$. 
In particular, if the algorithm terminates after $m$ iterations, then for all $d \in D(p_1,s_1) \setminus E^{(m)}$ we have $A_{\mathbf{r}}(0,d)  < \frac{1}{2} d.$
\end{prop}

We implemented the algorithm in Mathematica 13 \cite{Mat}. The file containing the code and its application to our case is available at the repository \cite{Git}. We point out that in the implementation we added some modifications and arguments which do not affect the basic functionality of the algorithm:
\begin{itemize}
    \item a list of pairs $(p,s)$ to be skipped if they occur is some $K_m$;
    \item output of example admissible pairs $(i,j)$ along with corresponding $(p,s)$;
    \item the maximal accepted value of $s_2(j)$;
    \item a variant of Step 3, where in the case $|p|<|s|$ we instead add $(p0,s)$, $(p1,s)$ into $K_{m+1}$ (which usually simplifies the results).
\end{itemize}

\subsection{Application of the algorithm}
We initially applied the algorithm to the input $(p_1,s_1) = (1,1)$, obtaining admissible pairs in most cases. However, in certain ``branches'' $(p,s)$ the algorithm seemed not to terminate after finitely many iterations. For the sake of further analysis, we collect these problematic cases in the following way:
\begin{itemize}
    \item $(p,s) = (100,001)$,
    \item $(p,s) = (10^5,1^401)$,
    \item $(p,s) = (1u,1^5)$ for any $u \in \{0,1\}^4$.    
\end{itemize}
The reason why the algorithm fails in these cases is rather simple to explain. First, the set $D(100,001)$ contains infinitely many elements of the form $d=2^k+1$ for which $A_{\mathbf{r}}(0,d) <d/2$ is not satisfied, as we have already proved. In the second case, the pair $(i,j) = (2,1)$ does not satisfy condition (b) of Lemma \ref{lem:main_idea}, while all the other pairs do not satisfy (c) due to lack of control of carries. In the third case it is again not possible to satisfy (c). 

If we rerun the algorithm and demand that it ignores the branches listed above, then after $9$ iterations it terminates successfully, returning the set $S^{(9)}$ of over $500$ pairs $(p,s)$, and $E^{(9)} =\{1,5,7,39\}$. A file containing the code and detailed results is available at the repository \cite{Git}.
Summarizing, we get the following result.

\begin{prop} \label{prop:application_1}
    For all odd integers $d$ such that
    $$d \not\in  D(100,001) \cup D(1,1^5) \cup D(10^5,1^401) \cup \{1,5,7,39\} $$
    we have $A_{\mathbf{r}}(0,d) < d/2$.
\end{prop}

We note that the exceptions $d=1,5,7$ are already covered by cases (i) and (ii) of Theorem \ref{thm:main1}. Hence, $d=39$ is the only ``true'' exception.

We now deal with the remaining cases, starting with $(p,s) = (100,001)$.

\begin{prop} \label{prop:problematic1}
For all $d \in D(100,001) \setminus \{2^k+1: m \geq 3\}$ we have $r_{(2^{\ell-1}+1)d}=1$. Consequently, $A_{\mathbf{r}}(0,d) < d/2$.
\end{prop}
\begin{proof}
Write $d = [w001]_2 = [100u]_2$ for some $w,u \in \{0,1\}^*$. Then $(2^{\ell-1}+1)_2 \cdot (d)_2$ yields

$$ \begin{array}{lllllllll}
        &   &   & 1 & 0  & 0 & u  \\
+       & w & 0 & 0 & 1  &   &    \\
\hline  & w & 0 & 1 & 1  & 0 & u         
\end{array}, $$
which implies $A_{\mathbf{r}}(0,d) \leq 2^{\ell-1}+1$. If $d$ is not of the form $2^k+1$, then $d/2 > 2^{\ell-1}+1$, and the result follows.
\end{proof}

We move on to the case $(p,s) = (10^5,1^401)$.

\begin{prop} \label{lem:problematic3}
For all $d \in D(10^5,1^401)$ and either $n = 2^{\ell-4}+1$ or $n = 2^{\ell-5}+1$ we have $r_{nd} = 1$. Consequently, $A_{\mathbf{r}}(0,d)<d/2$.
\end{prop}
\begin{proof}
Write $d = [10^5u]_2 = [w1^401]_2$ and let $v$ denote the result of binary addition $w + 1$.
If $n = 2^{\ell-4}+1$, then $(n)_2 \cdot (d)_2$ yields

$$ \begin{array}{lllllllllll}
       &    &   &   & 1 & 0  &  0 & 0 & 0 & 0 & u \\
+      &  w & 1 & 1 & 1 & 1  &  0 & 1 &   &   &   \\
\hline &  v & 0 & 0 & 0 & 1  &  0 & 1 & 0 & 0 & u        
\end{array}, $$
while for $n = 2^{\ell-4}+1$ we obtain
$$ \begin{array}{lllllllllll}
       &    &   & 1 & 0 & 0  &  0 & 0 & 0 & u  \\
+      &  w & 1 & 1 & 1 & 1  &  0 & 1 &   &    \\
\hline &  v & 0 & 0 & 1 & 1  &  0 & 1 & 0 & u       
\end{array}. $$
By comparing the result of both operations, we see that the parity of $e_{11}(nd)$ differs depending on the choice of $n$. Clearly, we have $n < d/2$, and the result follows.
\end{proof}

We now deal with the remaining set $D(1,1^5)$. We will use the following lemma to filter out most of its elements, leaving only a few subcases to consider.

\begin{lem} \label{lem:suffix_reduction}
Let $k, i$ be integers such that $2 \leq i < k$. Let $p \in \{0,1\}^{*}$ begin with a $1$ and assume that $(i,1)$ is an admissible pair for $(p,a01^k)$, where $a \in \{0,1\}$. Then $(i,1)$ is also an admissible pair for $(p,a01^{k+2l})$ for any $l \geq 0$.
\end{lem}
\begin{proof}
Condition \eqref{eq:A1} clearly holds for $s = a01^{k+2l}$ so it remains to check \eqref{eq:A2}. Take $d \in D(p,a01^k)$ and let $d=[wa01^k]_2 = [pu]_2$ for some $w,u \in \{0,1\}^*$. Also write $p = p_1 p_2$, where $|p_1| = i$ and $|p_2| = |p|-i > 0$  (recall that admissibility entails $i < |p|$). Then binary multiplication $(2^{\ell-i}+1)_2 \cdot (d)_2$ yields
$$ \begin{array}{lllllllllll} 
       &    &   &   &         & p_1  &  p_2 & u  \\
+      &  w & a & 0 & 1^{k-i} & 1^i  &      &     \\
\hline &  w & a & 1 & 0^{k-i} & v    &  p_2 & u        
\end{array}, $$
where $1v$ is the result of binary addition $p_1 + 1^i$. Then \eqref{eq:A2} for $s= a01^k$ becomes
\begin{equation} \label{eq:admissible_2}
    1 \equiv |pu|_{11} + (|wa|_{11}+k-1) + (|wa|_{11} + a+ |v p_2 u|_{11}) \equiv k+a-1+|vp_2|_{11}-|p|_{11}  \pmod{2},
\end{equation}
where we used  equality
$$|vp_2u|_{11} - |vp_2|_{11} =  |p_2u|_{11} - |p_2|_{11} = |p_1 p_2u|_{11} -|p_1p_2|_{11} = |pu|_{11} -|p|_{11}.$$
Now, let $d' = [w'a01^{k+2m}]_2 = [pu']_2 \in D(p,a01^{k+2l})$. After performing the same operations we again arrive at the condition \eqref{eq:admissible_2}, where $k$ is replaced by $k+2l$. But this does not change parity so the result follows.
\end{proof}

In order to apply the lemma, we again run our algorithm separately four times for inputs of the form $(1,a01^k)$, where $a \in \{0,1\}$ and $k=5,6$. Here we also impose the following conditions:
\begin{itemize}
  \item $s_2(j) = 1$, i.e., $j=1$ in order to obtain admissible pairs $(i,1)$;
  \item limit the number of iterations to at most $5$ to guarantee $i < |p| \leq 5$.
\end{itemize}
Collecting the results of all four runs, we obtain 28 cases $(p,s)$ with admissible pairs of the form $(i,1)$, where $2 \leq i<5$. The exact list together with the corresponding $i$ can be found in the supplemental files. On the other hand, there remain 8 unresolved cases $(p,s)$, given in Table \ref{tab:no_admissible} below.

\begin{table}[H] 
\begin{tabular}{l|l|l}
case & $p$ & $s$                   \\ \hline
I & $10010$         & $001^{5}$   \\
II & $11111$        & $001^{5}$            \\
III & $10110$         & $101^{5}$        \\
IV & $11001$         & $101^{5}$   \\
V & $10110$         & $001^{6}$  \\
VI & $11001$         & $001^{6}$  \\
VII & $10010$         & $101^{6}$  \\
VIII & $11111$         & $101^{6}$  
\end{tabular}
\caption{The remaining cases}
\label{tab:no_admissible}
\end{table}

Let $R$ denote the set of $(p,s)$ listed in the table. By virtue of Lemma \ref{lem:suffix_reduction} and the assumption that $(d)_2$ contains at least one $0$, we thus get the following result.

\begin{prop}  \label{prop:application_2}
    For all 
    $$d \in D(1,1^5) \setminus \left(\{2^k-1: k \geq 5\} \cup \bigcup_{(p,s) \in R} \bigcup_{l = 0}^{\infty} D(p,s1^{2l}) \right)$$ we have $A_{\mathbf{r}}(0,d) < d/2$.
\end{prop}

We now deal with the remaining cases.
\begin{prop}
    For all $(p,s) \in R$, $l \geq 0$ and $d \in D(p,s1^{2l})$ we have $A_{\mathbf{r}}(0,d) < d/2$.
\end{prop}
\begin{proof}
    We refer to the case numbers according to Table \ref{tab:no_admissible} (where we append $1^{2l}$ to $s$).
    
    First, we claim that for each of the cases III--VI, the pair $(i,j)=(4,5)$ is admissible. Since the computations are very similar in each case, we only show this for case III. First, we can quickly check that if $d \in D(10110,101^{5+2l})$, then $(5d)_2$ has the suffix $101^{2+2l}011$ of length $7+2l$. If we write $5d =[w101^{2+2l}011]_2$ and $d=[10110u]$, then binary addition $(2^{\ell-4} \cdot 5 d)_2 + (d)_2$ gives
    $$ \begin{array}{lllllllllll}
       &   &   &   &   & 1 & 0 & 1 & 1 & 0 & u \\
+      & w & 1 & 0 & 1^{1+2l} & 1 & 0  & 1 & 1 & & \\
\hline & w & 1 & 1 & 0^{1+2l} & 0 & 1 & 1 & 0 & 0 & u
\end{array}.$$
Adding the number occurrences of $11$ in each line we get
$$ (1+|u|_{11}) + (|w1|_{11} + 2+2l) + (|w1|_{11} + 2+|u|_{11}) \equiv 1 \pmod{2}. $$
At the same time, $\ell(2^{\ell-4} \cdot 5 +1) = \ell -1$ and $(2^{\ell-4} \cdot 5 +1)_2$ has a prefix $1010$ lexicographically smaller than $1011$, which means that $2^{\ell-4} \cdot 5 +1 < d/2$.

We move on to cases I and VII which are again treated in a similar fashion, thus we focus on the former one. This time we modify our method slightly and prove that if $d \in D(10010,001^{5+2l})$, then at least one of $r_{d}, r_{3d}, r_{(2^{\ell-3}+3)d}$ equals $1$. First, $11011$ is the prefix of $(3d)_2$ of length $5$. If we write $d =[w001^{5+2l}]_2$ and $3d=[11011u]_2$, then binary addition $(2^{\ell-3}d)_2 + (3d)_2$ gives
$$ \begin{array}{lllllllllll}
       &   &   &   &   & 1 & 1 & 0 & 1 & 1 & u \\
+      & w & 0 & 0 & 1^{1+2l} & 1 & 1  & 1 & 1 &  & \\
\hline & w & 0 & 1 & 0^{1+2l} & 1 & 1 & 0 & 0 & 1 & u
\end{array}.$$
Adding the number occurrences of $11$ in each line we get
$$ (2+ |1u|_{11}) +  (|w|_{11} + 4+2l) + (|w|_{11}+1+|1u|_{11}) \equiv 1 \pmod{2}. $$
This time $\ell(2^{\ell-3} + 3) = \ell - 2$ so we immediately get $2^{\ell-3} + 3 < d/2$.

We move on to the hardest two cases:  II and VIII. Starting with case II, we need consider a few further possibilities, concerning the prefix and suffix $p,s$ of $(d)_2$.

\textbf{Case IIa}: $p = 1^{5+2l}00$ and $s=0001^{5+2l}$.

We claim that at least one of $r_{d}, r_{(2^{3+2l}+1)d}, r_{(2^{\ell-2}+2^{3+2l}+1)d}$ equals $1$. First we compute how the prefix of $((2^{3+2l}+1)d)_2$ of length $8+2l$ looks like:
$$\begin{array}{llllllll}
       &   &         & 1 & 1       & 1 & 1 & \cdots \\
+      &   & 1^{3+2l} & 1 & 1       & 0 & 0 & \cdots  \\
\hline & 1 & 0^{3+2l} & 1 & \bar{c} & c & c & u      
\end{array},$$
where $c \in \{0,1\}$ and $u$ is a suffix of the result. The value of $c$ depends on whether there is a carry added directly to the left of $u$. Now,  write $d = [w0001^{3+2l}]_2$. For either value of $c$, binary addition $((2^{3+2l}+1)d)_2 + (2^{\ell-2}d)_2$ gives
$$\begin{array}{lllllllllll}
       &   &   &   & 1 & 0^{3+2l} & 1       & \bar{c} & c & c & u \\
+      & w & 0 & 0 & 0 & 1^{3+2l} & 1       & 1       &   &   &   \\
\hline & w & 0 & 1 & 0 & 0^{3+2l} & \bar{c} & c       & c & c & u
\end{array}.$$
Adding the number occurrences of $11$ in each line we get
$$  (1 + |cu|_{11}) + (|w|_{11} + 4+2l) +  (|w|_{11} + 2c + |cu|_{11}) \equiv 1 \pmod{2}.$$ We have $\ell(2^{\ell-2}+ 2^{3+2l}+1) = \ell-1$ and the binary expansion $(2^{\ell-2}+ 2^{3+2l}+1)_2$ starts with $10$, which means that $2^{\ell-2}+ 2^{3+2l}+1 < d/2$.

\textbf{Case IIb}: $p = 1^{5+2l}00$ and $s=1001^{5+2l}$.

A quick calculation shows that the pair $(i,j) = (6+2l,1)$ is admissible.

\textbf{Case IIc}: $p = 1^{a}0$ and $s=001^{5+2l}$, where $a > 5+2l$.

We claim that for either $n = 2^{\ell-4-2l}+1$ or $n = 2^{\ell-5-2l}+1$ we have $r_{nd} = 1$.
Write  $d = [w 00 1^{5+2l}]_2 = [1^a0u] = $ for some $u,w \in \{0,1\}^*$.
In the first case $(n)_2 \cdot (d)_2$ yields
$$ \begin{array}{lllllllll}
       &    &   &   & 1^{3+2l} & 1  & 1^{a-4-2l} & 0 & u \\
+      &  w & 0 & 1 & 1^{3+2l} & 1  &           &   &   \\
\hline &  w & 1 & 0 & 1^{3+2l} & 0  & 1^{a-4-2l} & 0 & u         
\end{array}, $$
while in the second case we get
$$\begin{array}{llllllllll}
       &   &   & 1^{4+2l} & 1   & 1^{a-5-2l} & 0 & u \\
+      & w & 0 & 1^{4+2l} & 1   &         &   &   \\
\hline
	   & w & 1 & 1^{4+2l} & 0   & 1^{a-5-2l} & 0 & u     
\end{array}.$$
We see that the parity of $e_{11}(nd)$ differs depending on the choice of $n$, and our claim follows (it is clear that $n < d/2$).

\textbf{Case IId}: $p = 1^{a}0$ and $s=001^{5+2l}$, where $l \geq 1$ and $5 \leq a < 5+2l$.

In a similar fashion as in the previous case one can show that for either $n = 2^{\ell-a+1}+1$ or $n = 2^{\ell-a}+1$ we have $r_{nd} = 1$.

Finally, we consider case VIII, which is again split into further subcases.

\textbf{Case VIIIa}: $p = 1^{6+2l}00$ and $s=101^{6+2l}$.

We will show that $(i,j) = (7+2l, 2^{6+2l}-1)$ is an admissbile pair.
We first determine the suffix of $((2^{6+2l}-1)d)_2$ of length $8+2l$ through the subtraction $(2^{6+2l}d)_2 - (d)_2$.  We have
$$ \begin{array}{llllll}
	   & \cdots & 1 & 1 & 0^{5+2l} & 0 \\
 -     & \cdots  & 1 & 0 & 1^{5+2l} & 1 \\     
\hline & w & 0 & 0 & 0^{5+2l} & 1
\end{array},$$ 
where $w$ is a prefix of $((2^{6+2l}-1)d)_2$. Now,  write $d = [1^{6+2l}00u]_2$. Binary addition $(d)_2+(2^{\ell-7-2l}(2^{6-2l}-1)d)_2$ gives
$$\begin{array}{lllllll}
       &   &   & 1^{6+2l}  & 0 & 0 & u  \\
+      & w & 0 & 0^{6+2l} & 1 &  &     \\
\hline & w & 0 & 1^{6+2l} & 1 & 0 & u
\end{array}.$$
By summing the number of occurrences of $11$ in each line, we obtain 
$$(5+2l +|u|_{11}) + |w|_{11} + (|w|_{11}+6+2l+|u|_{11}) \equiv 1 \pmod{2}.$$ 
Verification of the inequality $A_{\mathbf{r}}(0,d) <d/2$ requires a bit more care compared to the previous cases.
We have $\ell(2^{\ell-7-2l}(2^{6+2l}-1) +1) = \ell-1$. Since $p$ and $s$ cannot overlap in $(d)_2$, we must have $(d)_2 = 1^{6+2l}00v101^{6+2l}$ for some $v \in \{0,1\}^*$. The result now follows from the fact that the expansion $(2^{\ell-7-2l}(2^{6+2l}-1) +1)_2 = 1^{6+2l} 0^{\ell-8-2l} 1$ is lexicographically smaller than the prefix of $(d)_2$ of length $\ell-1$, regardless of $v$.

\textbf{Case VIIIb}: $p = 1^{6+2l}01$ and $s=101^{6+2l}$.

The idea is similar as in the case IIa. We claim that at least one of $r_{d}, r_{5d}, r_{(2^{\ell-3-2l}+5)d}$ equals $1$. We have $(5d)_2 = 1 0 0 1^{4+2l}  0  c  u $, where $c \in \{0,1\}$ and $u$ is a suffix of $(5d)_2$. Writing $d = [w101^{6+2l}]_2$, binary addition $(5d)_2 + (2^{\ell-3-2l}d)_2$ becomes
$$
\begin{array}{llllllllllll}
       &   &   &   & 1 & 0 & 0 & 1^{2+2l} & 1 & 1 & c & u \\
+      & w & 1 & 0 & 1 & 1 & 1 & 1^{2+2l} & 1 &   &   &   \\
\hline & w & 1 & 1 & 1 & 0 & 0 & 1^{2+2l} & 0 & 1 & c & u
\end{array}.
$$
The sum of the number of occurrences of $11$ in each line satisfies 
$$ (3+2l+c + |cu|_{11}) + (|w1|_{11}+ 5+2l) + (|w1|_{11}+3+2l+c+|cu|_{11}) \equiv 1 \pmod{2}.  $$
It is simple to check that $2^{\ell-3-2l}+5<d/2$ so our claim holds.

\textbf{Case VIIIc}: $p = 1^{a}0$ and $s=101^{6+2l}$, where $a > 6+2l$.

Similarly as in Case IIc, for either $n = 2^{\ell-5-2l}+1$ or $n = 2^{\ell-6-2l}+1$ we have $r_{nd} = 1$.

\textbf{Case VIIId}: $p = 1^{a}0$ and $s=101^{6+2l}$, where $5 \leq a < 6+2l$.

Similarly as in Case IId, for either  $n = 2^{\ell-a+1}+1$ or $n = 2^{\ell-a}+1$ we have $r_{nd} = 1$.
\end{proof}

\subsection{Finishing the proof}

To conclude this section, we summarize the above results and show that part (iii) of Theorem \ref{thm:main1} holds, thus completing the proof of the theorem.

\begin{prop}
For all integers $d \geq 3$ not lying in the set
$$ \{ 2^k-1: k \geq 3 \text{ is odd} \} \cup \{ 2^k+1: k \geq 2   \} \cup \{39\} $$
we have $A_{\mathbf{r}}(d) < d/2$.
\end{prop}
\begin{proof}
The propositions proved in this section collectively show that the claim is true for all odd integers $d \geq 3$, except for possibly $d = 2^k-1$ with $k \geq 2$ even. But for such $d$ we get $A_{\mathbf{r}}(0,d) = 1 < d/2$ as well. 

For even $d$, we use the equality $A_{\mathbf{r}}(0,2d) = A_{\mathbf{r}}(0,d)$. We have $A_{\mathbf{r}}(0,2^k) = A_{\mathbf{r}}(0,1) = 3,$ which gives the desired inequality for any $k \geq 2$. When $d = 2^k d'$, where $k \geq 1$ and $d' \geq 3$ is odd, we get for $d > 6$ the inequality
$$ A_{\mathbf{r}}(0,d) = A_{\mathbf{r}}(0,d') \leq \frac{1}{2}(d'+3) \leq \frac{d}{4} + \frac{3}{2} < \frac{d}{2}.$$
Finally, $ A_{\mathbf{r}}(0,6) = 1$ and the proof is finished.
\end{proof}

\section{Proof of Theorem \ref{thm:main2} and Theorem \ref{thm:main3}} \label{sec:main2_main3}

In this section we consider the values $A_{\mathbf{r}}(d)$ when $d=2^k+1$ or $d=2^k-1$. To begin, we make a few simple observations. Let $n \in \N$ and write 
$$n = 2^k m + l,$$ 
where $m \in \N$ and $0 \leq l < 2^{k}$. Adding $d = 2^k+1$ to $n$ amounts to increasing both $m$ and $l$ by $1$, unless $l = 2^k-1$. Similarly, adding $d = 2^k-1$ corresponds to increasing $m$ and decreasing $l$ by $1$, unless $l=0$.

At the same time, we have
\begin{equation} \label{eq:RS_identity}
 r_{n}\equiv \begin{cases} 
r_{m} + r_{l} \pmod{2}  &\text{if } l <2^{k-1}, \\
 r_{2m+1} + r_{l} \pmod{2}  &\text{if } l \geq 2^{k-1}.
\end{cases}
\end{equation}

We would like to show that the existence of a long progression imposes some relation on $m$ and $l$. Suppose that $A_{\mathbf{r}}(n,2^k+1) \geq a + 1$  for some $a \in \N$. If we additionally assume that $0 \leq l < 2^{k-1}-a$, then by \eqref{eq:RS_identity} we obtain  $r_m r_{m+1}\cdots r_{m+a}=r_l r_{l+1}\cdots r_{l+a}$. Similarly, if $2^{k-1} \leq l < 2^{k}-a$, then we get $r_{2m+1} \cdots r_{2(m+a)+1} =r_{2l+1} \cdots r_{2(l+a)+1}$.  In the same fashion, when considering a progression of difference $d=2^k-1$, under suitable assumptions we arrive at the equalities $r_{m}\cdots r_{m+a}=(r_{l-a} \cdots r_l)^R$ and $r_{2m+1} \cdots r_{2(m+a)+1}=(r_{l-a} \cdots r_l)^R$.

This leads to the study of subwords of $\mathbf{r}$ and its subsequence $(r_{2m+1})_{m \geq 0}$, as well as their reversals.  This is conveniently done using the description of the Rudin--Shapiro sequence in terms of a substitution $\rho$, which we now recall. Let $\rho$ be defined on the alphabet $\Sigma_{\mathbf{r}} = \{00,01,10,11\}$ by
$$ \rho(00)=00 \, 01, \qquad \rho(01)=00 \,10 \qquad \rho(10)=11 \, 01, \qquad \rho(11)=11 \, 10.$$
Letting $\rho(vw) = \rho(v)\rho(w)$ for $v,w \in \Sigma_{\mathbf{r}}$, we can extend $\rho$ to all finite and infinite words over $\Sigma_{\mathbf{r}}$. Then $\mathbf{r}$ can be identified with the fixed point of $\rho$ starting with $00$, where each block in $\Sigma_{\mathbf{r}}$ is split into two individual letters. More precisely, for $t \in \N$ let $\rho^t$ be the $t$-fold composition of $\rho$ with itself, where $\rho^{0}$ is the identity. Then
$$ \mathbf{r} = \lim_{t \to \infty} \rho^t(00) = 00 \, 01 \, \rho(01) \, \rho^2(01) \cdots. $$
Here the limit means that each of the words $\rho^t(00)$ is a prefix of the infinite word $\mathbf{r}$ and their length tends to infinity. For each $t \in \N$ we can also write
\begin{equation} \label{eq:aligned_blocks}
\mathbf{r} = \rho^t(00) \rho^t(01) \rho^t(00) \rho^t(10) \cdots.
\end{equation}
For $w \in \Sigma_{\mathbf{r}}$ each block $\rho^t(w)$, treated as a word over $\{0,1\}$, has length $2^{t+1}$ and we call it a $2^{t+1}$-aligned block. Hence, for every $t \in \N$ there are precisely four distinct $2^{t+1}$-aligned blocks. Note that the blocks $\rho^t(00),\rho^t(11)$ appear at even positions (counting from $0$) in the representation \eqref{eq:aligned_blocks}, while the blocks $\rho^t(01),\rho^t(10)$ -- at odd positions. Also observe that for any $w \in \Sigma_{\mathbf{r}}^*$ we have $\rho(\overline{w}) = \overline{\rho(w)}$.

It is useful to have a similar description in terms of a substitution for the subsequence $(r_{2n+1})_{n \geq 0}$. For $n \in \N$ we put $s_n = r_{2n+1}$ and $\mathbf{s} = (s_n)_{n \geq 0}$. We obtain a nice property linking $2^{t+1}$-aligned subwords of $\mathbf{r}$ and $\mathbf{s}$.

\begin{prop} \label{prop:interesting_property}
The sequence $\mathbf{s}$ is the fixed point starting with $01$ of the substitution $\sigma$, defined by
$$ \sigma(01) = 01 \: 00, \quad \sigma(00) = 01 \: 11, \quad  \sigma(11) = 10 \: 00, \quad \sigma(10) = 10 \: 11.$$
Moreover, for any odd integer $t \geq 1$ we have
$$ \sigma^t(01) = \rho^t(01)^R, \quad \sigma^t(00) = \rho^t(11)^R, \quad \sigma^t(11) = \rho^t(00)^R, \quad  \sigma^t(10) = \rho^t(10)^R,  $$
while for any even integer $t \geq 2$ we have
$$ \sigma^t(01) = \rho^t(10)^R, \quad \sigma^t(00) = \rho^t(00)^R, \quad \sigma^t(11) = \rho^t(11)^R, \quad  \sigma^t(10) = \rho^t(01)^R.  $$
\end{prop}
\begin{proof}
The word $\mathbf{r}$ is the fixed point starting with $0001$ of the following substitution, derived from $\rho$:
$$ 0001 \mapsto 0001 \: 0010, \quad 0010 \mapsto 0001 \: 1101, \quad  1101 \mapsto 1110 \: 0010, \quad 1110 \mapsto 1110 \: 1101.$$
Note that the first and third digit (counting from $0$) uniquely determine each block of length $4$. By extracting the digits at odd positions from the above substitution, we obtain precisely $\sigma$. But this corresponds to taking the subsequence $(r_{2n+1})_{n \geq 0}$, and thus we get the first part of the assertion.

In order to prove the identities linking $\rho$ and $\sigma$, we use induction on $t$. For $t = 1$ the result follows immediately from the definition of $\sigma$. Now let $t \geq 2$ be even. By the inductive assumption we have for example
\begin{align*} 
\sigma^t(01) &= \sigma^{t-1}(01) \sigma^{t-1}(00) = \rho^{t-1}(01)^R \rho^{t-1}(11)^R \\
 &= (\rho^{t-1}(11) \rho^{t-1}(01) )^R = \rho^t(10)^R,
\end{align*}
as claimed. The verification for other blocks and even values of $t$ is similar. 
\end{proof}

We now prove a few (standard) auxiliary results concerning the subwords of the infinite words $\mathbf{r}$ and $\mathbf{s}$. Lemmas \ref{lem:negation}--\ref{lem:uniquely_determined_position} have the same statement and almost identical proof for both $\mathbf{r}$ and $\mathbf{s}$. Hence, in these lemmas we let $\mathbf{u}$ denote any of $\mathbf{r},\mathbf{s}$ and only provide the proof for $\mathbf{u} = \mathbf{r}$.  The first result says that the set of subwords of $\mathbf{u}$ is invariant with respect to negation.

\begin{lem} \label{lem:negation}
We have $w \in \Sub(\mathbf{u})$ if and only if $\overline{w} \in \Sub(\mathbf{u})$. 
\end{lem}
\begin{proof}
If $w \in \Sub(\mathbf{r})$, then $w$ is a subword of $\rho^t(00)$ for some $t \geq 1$, and thus a subword of  $\rho^{t+3}(11) = \overline{\rho^{t+3}(00)}$. This means that $\overline{w} \in \Sub(w)$. 
The other inclusion follows by $\overline{\overline{w}} = w$. 
\end{proof}

The next lemma shows that from the knowledge of a sufficiently long prefix of $\mathbf{u}$ we can infer all the subwords of given length and their positions modulo a power of $2$.

\begin{lem} \label{lem:appearance_of_subwords}
Let $w \in \Sub(\mathbf{u})$ have length $|w| \leq 2^t +1$, where $t \geq 1$ is an integer. Then $w \in \Sub(p)$, where $p$ is the prefix of $\mathbf{u}$ of length $7 \cdot 2^{t+1}$. More precisely, if $w$ appears in $\mathbf{u}$ at some position modulo $2^{t+1}$, then it appears in $p$ at the same position modulo $2^{t+1}$.
\end{lem}
\begin{proof}
Divide $\mathbf{r}$ into $2^t$-aligned blocks:
\begin{equation} \label{eq:Rudin_Shapiro_blocks}
\mathbf{r} = \rho^{t-1}(00) \rho^{t-1}(01) \rho^{t-1}(00) \rho^{t-1}(10) \cdots.
\end{equation}
Any occurrence of a subword $w$ of length an most $2^t+1$ is contained in a concatenation of $2^t$-aligned blocks $\rho^{t-1}(A) \rho^{t-1}(B)$, where one of $A,B$ belongs to the set $\{00,11\}$, and the other to $\{01,10\}$. Hence, in order to prove both parts of the claim it is sufficient to show that each possible concatenation $\rho^{t-1}(A) \rho^{t-1}(B)$ appears within $14$ initial terms of \eqref{eq:Rudin_Shapiro_blocks}.
In the case $t=1$ we can check this by inspection of the initial $2$-aligned blocks:
$$00 \: 01 \: 00 \: 10 \: 00 \: 01 \: 11 \: 01 \: 00 \: 01 \: 00 \: 10 \: 11 \: 10.$$
For $t \geq 2$ the result follows by applying $\rho^{t-1}$ to this word block-by-block.
\end{proof}

We now show that the positions at which the same word $w$ occurs in $\mathbf{u}$ are unique modulo a power of $2$. Note that unlike in the previous lemma, the length of $w$ is also bounded from below.

\begin{lem} \label{lem:uniquely_determined_position}
Let $w \in \Sub(\mathbf{u})$ have length $|w| \geq 9$ and let $t \geq 3$ be the integer such that $2^t+1 \leq |w| \leq 2^{t+1}$.
If $w$ is a subword of a $2^{t+1}$-aligned block in $\mathbf{u}$, then its position modulo $2^{t+2}$ in $\mathbf{u}$ is uniquely determined.
\end{lem}
\begin{proof}
We prove the result by induction on $t$. In the base case $t=3$,  Lemma \ref{lem:appearance_of_subwords} guarantees that if $w$ occurs at some position modulo $2^5$ in $\mathbf{r}$, then it already occurs at the same position modulo $2^5$ in the prefix $p$ of $\mathbf{r}$ of length $|p|=7 \cdot 2^5$. A computer search shows that for each subword $w \in \Sub(p)$ of suitable length there is exactly one such position.

Now let $t \geq 4$ and express $\mathbf{r}$ as a concatenation of $2^{t}$-aligned blocks:
$$\mathbf{r} = \rho^{t-1}(00) \rho^{t-1}(01) \cdots =  R_0 R_1 \cdots.$$
By the assumption $w$ is a subword of a $2^{t+1}$-aligned block $R_{2n} R_{2n+1}$ for some $n \in \N$. We can write $w = uv$, where $u \in \Suf(R_{2n})$ and $v \in \Pref(R_{2n+1})$. Hence, each of $u,v$ is a subword of a $2^t$-aligned word and at least one of them has length greater than or equal $2^{t-1}+1$. In either case it follows from the inductive assumption that the position of $w$ modulo $2^{t+1}$ in $\mathbf{r}$ is uniquely determined. (This also shows that the factorization $w = uv$ is unique.) 

In order to determine the position of $w$ modulo $2^{t+2}$ it remains to deduce the parity of $n$. 
We have $R_{2n} = \rho^{t-1}(A)$ and $R_{2n+1} = \rho^{t-1}(B)$  for some $A \in \{00,11\}, B \in \{01,10\}$. Because both words $u,v$ are nonempty, at least the last letter of $\rho^{t-1}(A)$ and the first letter of $\rho^{t-1}(B)$ are known. As $\rho^{t-1}(11) = \overline{\rho^{t-1}(00)}$ and $\rho^{t-1}(10) = \overline{\rho^{t-1}(01)}$, this already allows us to determine $A$ and $B$. Let $C \in \Sigma_{\mathbf{r}}$ be such that $\rho(C) = AB$. Then $R_{2n} R_{2n+1} = \rho^t(C)$ and thus $C$ is the $n$th subsequent $2$-aligned block in $\mathbf{r}$ (counting from $0$). Hence, $n$ is even if $C \in \{00,11\}$, and odd otherwise.
\end{proof}

The following two lemmas provide some relations between the subwords of $\mathbf{r}$ and $\mathbf{s}$. First, we show that the subwords of $\mathbf{r}$ are precisely the reversed subwords of $\mathbf{s}$ and relate the positions at which $w$ and $w^R$ can appear in the respective sequences.

\begin{lem} \label{lem:reversed_subwords}
We have $w \in \Sub(\mathbf{r})$ if and only if $w^R \in \Sub(\mathbf{s})$.
Moreover, assume that $w$ has length $|w| \geq 2^t+1$ and is a subword of a $2^{t+1}$-aligned block in $\mathbf{r}$ for some integer $t \geq 3$. If $w = r_n r_{n+1} \cdots r_{n+|w|-1}$ and $w^R = s_m s_{m+1} \cdots s_{m+|w|-1}$ for some integers $n,m \geq 0$, then
\begin{equation} \label{eq:reversed_subwords}
n + m + |w| \equiv 0 \pmod{2^{t+2}}.
\end{equation}
\end{lem}
\begin{proof}
First, if $w \in \Sub(\mathbf{r})$, then it is a subword of some $2^{t+1}$-aligned block in $\mathbf{r}$. Proposition \ref{prop:interesting_property} shows that $2^{t+1}$-aligned blocks in $\mathbf{s}$ are precisely reversals of $2^{t+1}$-aligned blocks, hence $w^R \in \Sub(\mathbf{r})$. The converse is proved in the same fashion.

We proceed to the second part of the statement. Write $\mathbf{r}$ and $\mathbf{s}$ as the concatenation of $2^{t+1}$-aligned blocks:
\begin{align*}
\mathbf{r} &= R_0 R_1 R_2 \cdots,  \\
\mathbf{s} &= S_0 S_1 S_2 \cdots. 
\end{align*}
Let $w = r_n r_{n+1} \cdots r_{n+|w|-1} \in \Sub(R_i)$ for some $i \in \N$.
Write $n = 2^{t+1} i +l$, where $l < 2^{t+1}$ is the position of the initial letter of $w$ in $R_i$ (counting from $0$). For some $j \in \N$ we have $S_j = R_i^R$, and thus $w^R \in \Sub(S_j)$. More precisely, we have $w^R = s_m s_{m+1} \cdots s_{m+|w|-1}$, where $m = 2^{t+1} (j+1) - l - |w|$. The relations of Proposition \ref{prop:interesting_property} imply that $i,j$ have different parity, which gives the congruence \eqref{eq:reversed_subwords} for these particular occurrences of $w$ and $w^R$. But Lemma \ref{lem:uniquely_determined_position} implies that the positions at which $w$ and $w^R$ appear in $\mathbf{r}$ and $\mathbf{s}$, respectively, are unique modulo $2^{t+2}$, and therefore \eqref{eq:reversed_subwords} holds in general.
\end{proof}

The final auxiliary result shows in particular that $\mathbf{r}$ and $\mathbf{s}$ have only finitely many common subwords.

\begin{lem} \label{lem:common_subwords}
The following conditions are equivalent:
\begin{enumerate}[label={\textup{(\roman*)}}]
\item $w \in \Sub(\mathbf{r}) \cap \Sub(\mathbf{s})$;
\item $w \in \Sub(\mathbf{r})$ and $w^R \in \Sub(\mathbf{r})$.
\end{enumerate}
Moreover, if any of the above holds, then $|w| \leq 14$.
\end{lem}
\begin{proof}
Due to Lemma \ref{lem:reversed_subwords} the conditions (i) and (ii) are equivalent. By Lemma \ref{lem:appearance_of_subwords} any $w \in \Sub(\mathbf{r})$ of length $|w|=15$ appears in the prefix of $\mathbf{r}$ of length $7 \cdot 2^5$. A computer search of this prefix shows that no such $w$ satisfies $w^R \in \Sub(\mathbf{r})$.
\end{proof}

We are now ready to prove Theorems \ref{thm:main2} and \ref{thm:main3}.

\begin{proof}[Proof of Theorem \ref{thm:main2}]
Let $d = 2^k+1$
We begin with the case $k \leq 3$, together with $k=4,5$, where we verify the assertion by direct calculation. Consider for example $k=3$. One can check that a monochromatic arithmetic progression of difference $9$ and length $9$ starts with $r_{28}$. Now, any monochromatic arithmetic progression of difference $9$ and length $10$ would have to be contained in a subword of $\mathbf{r}$ of length $(10-1) \cdot 9 + 1 = 82 \leq 2^7+1$. By Lemma \ref{lem:appearance_of_subwords} each such subword appears in the prefix of $\mathbf{r}$ of length $7 \cdot 2^{8}$. But a direct search shows that there is no such progression of length $10$ among these initial terms, and hence neither in $\mathbf{r}$.

The other cases with $k \leq 5$ can be proved in the same fashion. In Table \ref{tab:progressions1} below we provide the initial term of a monochromatic arithmetic progression of postulated length $A_{\mathbf{r}}(d)$ and the length of a prefix of $\mathbf{r}$, where a longer progression would have to appear (which is not the case). 

\begin{table}[h] 
\begin{tabular}{ccccc}
$k$ & $d$  & $A_{\mathbf{r}}(2^k+1)$ & initial term & prefix length \\ \hline
$1$ & $3$  & $5$    & $28$     & $7 \cdot 2^5$         \\
$2$ & $5$  & $6$    & $31$     & $7 \cdot 2^6$         \\
$3$ & $9$  & $9$    & $43$     & $7 \cdot 2^{8}$      \\
$4$ & $17$ & $10$   & $495$    & $7 \cdot 2^{9}$      \\
$5$ & $33$ & $18$   & $980$    & $7 \cdot 2^{11}$     
\end{tabular} 
\caption{\label{tab:progressions1}The location of monochromatic arithmetic progressions of length $A_{\mathbf{r}}(2^k+1)$ for $k \leq 5$}
\end{table}

We now move on to the general case $k \geq 6$. We first exhibit a monochromatic arithmetic progression of postulated length $A_{\mathbf{r}}(d) = 2^{k-1}+2$. Let $n = 2^{2k+1}- 2^k-1 =[1^{k} 0 1^{k}]_2$ so that $r_n =  0$. Then $n+d = 2^{2k+1} = [1 0^{2k+1}]_2$, thus also $r_{n+d} =  0$. Furthermore, for all $i=0,1,\ldots, 2^{k-1} $ we can write $n + (i+1)d = 2^{2k+1} + 2^k i  + i$ and it is easy to see that $e_{11}(n + (i+1)d) = 2 e_{11}(i) \equiv 0 \pmod{2}$. Therefore, the $2^{k-1} + 2$ terms $n, n+d, \ldots, n+ (2^{k-1}+1)d$ 
constitute a monochromatic arithmetic progression in $\mathbf{r}$. Note that it cannot be prolonged in any direction, as the binary expansions of the numbers
\begin{align*}
n-d &= 2^{2k+1}- 2^{k+1}- 2 =[1^{k-1} 0 1^{k} 0]_2, \\
n+ (2^{k-1}+2)d &= 2^{2k+1} + 2^{2k-1} + 2^k + 2^{k-1} + 1 = [101 0^{k-2} 11 0^{k-2} 1]_2
\end{align*}
both contain an odd number of occurrences of $11$.

We proceed to show that $A_{\mathbf{r}}(d) \leq 2^{k-1} + 2$ in general by giving a bound on $A_{\mathbf{r}}(n,d)$ for each $n$. By Lemma \ref{lem:negation} it is sufficient to consider only arithmetic progressions of ``color'' $0$. Thus, assume that $n \in \N$ is such that $r_{n} = 0$. As at the beginning of this section, write 
$$n = 2^k m + l,$$
where $m \in \N$ and $0 \leq l < 2^k$. We are going consider four cases, depending on which quarter of the interval $[0, 2^k)$ the number $l$ belongs to. 

\textbf{Case I:} $0 \leq l < 2^{k-2}$

For  $i=0,1,\ldots,2^{k-2}$  the blocks $(m+i)_2$ and $(l+i)_2$ in the binary expansion of $(n+i)_2$ are broken up by at least one zero, so we have
$$ r_{n+id} \equiv r_{m+i} + r_{l+i} \pmod{2}. $$
If $A(n,d) < 2^{k-2} + 1$, there is nothing to prove.  Otherwise, we get $r_{k+i} = r_{l+i}$ for $i=0,1,\ldots,2^{k-2}$. The word $r_l r_{l+1} \cdots r_{l+2^{k-2}}$ is a subword of length $2^{k-2}+1$ of the $2^{k-1}$-aligned block $r_0 r_{1} \cdots r_{2^{k-1}-1}$ so Lemma \ref{lem:uniquely_determined_position} implies $m \equiv l \pmod{2^k}$. Letting $m = 2^k j + l$, we get $n = 2^{2k} j + (2^k+1)l = 2^{2k} j + ld$.  Comparing the binary expansions 
\begin{align*}
(n + (2^{k-1}-l)d)_2 &= (j)_2 10^{k-1} 1 0^{k-1}, \\
(n + (2^{k-1}-l+1)d)_2 &= (j)_2 10^{k-2}1 1 0^{k-2} 1,
\end{align*}
we see that the latter contains exactly one more occurrence of the pattern $11$, regardless of $j$. Hence, the monochromatic arithmetic progression cannot be prolonged beyond $n + (2^{k-1}-l)d \leq n + 2^{k-1}d$, and therefore $A(n,d) \leq 2^{k-1} + 1$.

\textbf{Case II:} $2^{k-2} \leq l < 2^{k-1}$

Put $n' = n+(2^{k-1} - l)d =  2^k m' + 2^{k-1}$, where $m' = m+2^{k-1} - l+1$. For $i=0,1,\ldots, 2^{k-1} - 1$ the block $(m'+i)_2$ in the binary expansion of $(n'+id)_2$ is directly followed by a $1$, so we have
$$r_{n'+id} \equiv r_{2(m'+i)+1} + r_{2^{k-1}+i}\pmod{2}.$$
It follows that $r_{n'+id} = 0$ as long as $r_{2(k'+i)+1} = r_{2^{k-1}+i}$. By Lemma \ref{lem:common_subwords} the former equality does not occur for some $i \leq 14$. Hence, by the assumption $k \geq 6$ we get  
$$A(n,d) \leq  \frac{n'-n}{d} + A(n',d) \leq  2^{k-2} + 14 < 2^{k-1}.$$

\textbf{Case III:} $2^{k-1} \leq l < 3 \cdot 2^{k-2}$

For at least $i =0, 1, \ldots, 2^{k-2}$ we have 
$$r_{n+id} \equiv r_{2(m+i)+1} + r_{l+i} \pmod{2}.$$ 
Again, due to Lemma \ref{lem:common_subwords} the right hand-side cannot be constantly equal to $0$ for all  $i=0,1,\ldots,14$. Therefore, $A(n,d) \leq 14 < 2^{k-1}$.

\textbf{Case IV:} $3 \cdot 2^{k-2} \leq l < 2^{k}$

Put $n' = n+(2^{k} - l)d =  2^k m'$, where $m' = 2^{k} - l + 1$. If we assume that $A(n,d) \geq  2^{k-1} + 1$, then $A(n',d) \geq 2^{k-2} + 1$. Arguing precisely as in case I (where $n$ is replaced by $n'$), we obtain $m' \equiv 0 \pmod{2^k}$. Hence, we can write $m' = 2^{k} j$, so that $n' = 2^{2k} j$. Consider the binary expansions:
\begin{align*}
(n' - d)_2 &= (j-1)_2 1^{k-1} 0 1^k, \\
(n' - 2d)_2 &= (j-1)_2 1^{k-2} 0 1^k 0.
\end{align*}
The former expansion contains exactly one more occurrence of the pattern $11$, which implies $r_{n'-d} \neq r_{n'-2d}$. In view of our assumption $r_n = r_{n+d} = \cdots = r_{n'} = \cdots$, the only possibility is that $n = n'-d$. But from case I we already know that $A(n',d) \leq 2^{k-1} + 1$ so $A(n,d) \leq 2^{k-1} + 2$, as claimed.
\end{proof}

\begin{proof}[Proof of Theorem \ref{thm:main3}]
For $k \leq 5$ the claim can be verified numerically, in the same way as in the previous proof. In Table \ref{tab:progressions2} below we provide the initial term of a monochromatic arithmetic progression of postulated length $A_{\mathbf{r}}(d)$ and the length of a prefix of $\mathbf{r}$ which is sufficient to be checked for the existence of a longer monochromatic progression.
\begin{table}[h] 
\begin{tabular}{ccccc}
$k$ & $d$  & $A_{\mathbf{r}}(d)$ & initial term & prefix length \\ \hline
$1$ & $1$  & $4$    & $7$      & $7 \cdot 2^3$  \\
$2$ & $3$  & $5$    & $28$     & $7 \cdot 2^5$  \\
$3$ & $7$  & $9$    & $95$     & $7 \cdot 2^7$  \\
$4$ & $15$ & $10$   & $39$     & $7 \cdot 2^9$  \\
$5$ & $31$ & $19$   & $32$     & $7 \cdot 2^{11}$       
\end{tabular}
\caption{\label{tab:progressions2}The location of monochromatic arithmetic progressions of length $A_{\mathbf{r}}(2^k-1)$ for $k \leq 5$}
\end{table}

In what follows we assume that $k \geq 6$. We now show the existence of a monochromatic arithmetic progression of difference $d$ and desired length $A_{\mathbf{r}}(d)$ in the general case $k \geq 5$. Regardless of the parity of $k$, let $n = 2^{2k}+d = [1 0^{k} 1^k]$. For each $i=0,1, \ldots, 2^{k-1}-1$ we can write $(n+id)_2 = 1 0^q ((i+1)d)_2$ for some $q \geq 1$ (depending on $i$), which implies $r_{n + id} = r_{(i+1)d}$. Proposition \ref{prop:special_case_1} assures that the terms $n, n+d, \cdots, n + (2^{k-1}-1)d$ form a monochromatic arithmetic progression of ``color'' $(k-1) \bmod{2}$. Because $n+ 2^{k-1}d = 2^{2k} + 2^{2k-1} + 2^{k-1} - 1 = [11 0^k 1^{k-1}]_2$, we also have $e_{11}(n+ 2^{k-1}d) = k-1$ and the progression can be prolonged by yet another term. This proves that $A_{\mathbf{r}}(d) \geq 2^{k-1} + 1$. Moreover, if $k$ is odd, we obtain a stronger inequality $A_{\mathbf{r}}(d) \geq 2^{k-1} + 3$, as the monochromatic progression can be extended backwards, namely to $n-d = 2^{2k} = [10^{2k}]_2$ and $n-2d = 2^{2k} - 2^k +1 = [1^k 0^{k-1} 1]_2$.

We now prove that there do not exist longer monochromatic progressions of difference $d$. Just like in the proof of Theorem \ref{thm:main2}, we can restrict our attention to progressions of ``color'' $0$. Let $n \in \N$ be such that $r_n = 0$ and write $n = 2^k m + l$, where $m \in \N$ and $0 \leq l < 2^k$. Again, we consider four cases.

\textbf{Case I:} $0 \leq l < 2^{k-2}$

We argue that the inequality $A_{\mathbf{r}}(n,d) \geq 2^{k-1} + 2$ implies that $k$ is odd and also that $A_{\mathbf{r}}(n,d) \leq 2^{k-1} + 3$, which is sufficient to obtain the result.
Put $n' = n+(l+1)d = 2^k m' + (2^k-1)$, where $m' = m+l$. By our assumption we get $A(n',d) \geq  2^{k-2} + 1$, and consequently for at least $i = 0,1,\ldots, 2^{k-2}$ we have 
$$ 0 = r_{n'+id} \equiv r_{2(m'+i)+1} + r_{2^k-1 - i} \pmod{2}. $$
By Lemma \ref{lem:reversed_subwords} we must have 
$$ 0 \equiv m' + (2^k-1-2^{k-2}) + (2^{k-2}+1) \equiv m' \pmod{2^{k}}.$$
Letting $m' = 2^k j,$ we obtain $n' = 2^{2k} j + 2^k-1$ and it is easy to check that $k$ must be odd, as otherwise we would have $r_{n'} \neq r_{n'-d}$. 

At the same time, the number of occurrences of the pattern $11$ in the expansions
\begin{align*}
(n' + 2^{k-1}d)_2 &= (j)_2 1 0^{k} 1^{k-1}, \\
(n' + (2^{k-1}+1)d)_2 &= (j)_2 1 0^{k-2} 10 1^{k-2} 0,
\end{align*}
differs by $1$, so the monochromatic progression cannot exceed $n' + 2^{k-1}d$. Therefore, we obtain $A_{\mathbf{r}}(n',d) \leq 2^{k-1} + 1$.  It remains to show $n \in \{n'-d, n'-2d\}$, as then $A_{\mathbf{r}}(n,d) \leq 2 + A_{\mathbf{r}}(n',d) \leq 2^{k-1} + 3$. 

If $n=0$, then by definition we have $n = n'-d$. If $n > 0$, we have $j > 0$ so the indices $n' - 2d$ and $n' - 3d$ are positive. By analyzing the binary expansions 
\begin{align*}
(n' - 2d)_2 &= (j-1)_2 1^k 0^{k-1} 1, \\
(n' - 3d)_2 &= (j-1)_2 1^{k-1} 0^{k-1} 1 0,
\end{align*}
we deduce that $r_{n' - 2d} \neq r_{n' - 3d}$. It follows that $n \in \{n'-d, n'-2d\}$, or else the considered arithmetic progression would not be monochromatic.

\textbf{Case II:} $2^{k-2} \leq l < 2^{k-1}$

For at least $i = 0,1,\ldots, 2^{k-2}$ we have 
$$
r_{n+id} \equiv r_{k+i} + r_{l - i} \pmod{2}.
$$
As long as $r_{n+id}=0$, we obtain $r_{k+i} = r_{l - i}$ and Lemma \ref{lem:common_subwords} implies that $A_{\mathbf{r}}(n,d) \leq 14 < 2^{k-1}$.

\textbf{Case III:} $2^{k-1} \leq l < 3 \cdot 2^{k-2}$

The number $n' = n + (l-2^{k-1}+1)d = 2^k(m+l-2^{k-1}+1) + 2^{k-1} - 1$ falls under case II. As a consequence, we get 
$$A_{\mathbf{r}}(n,d) \leq l - 2^{k-1} + 1 + A_{\mathbf{r}}(n',d) \leq 2^{k-2} + 14 < 2^{k-1}.$$

\textbf{Case IV:} $3 \cdot 2^{k-2} \leq l < 2^{k}$

For at least $i = 0,1,\ldots, 2^{k-2}$ we have 
$$
r_{n+id} \equiv r_{2(m+i)+1} + r_{l - i} \pmod{2}.
$$
If we assume that $A_{\mathbf{r}}(n,d) \geq 2^{k-2}+1$, then Lemma \ref{lem:reversed_subwords} implies that we must have $m + l \equiv -1 \pmod{2^{k}}.$ Letting $m = 2^k j -l -1,$ we obtain $n = 2^{2k} j -2^k - (2^k-1)l = 2^{2k}j  - (l+1)d -1$. But then we get
\begin{align*}
(n +(l+1-2^{k-1})d)_2 &= (j-1)_2 1 0^{k} 1^{k-1}, \\
(n +(l+2-2^{k-1})d)_2 &= (j-1)_2 1 0^{k-2} 10 1^{k-2} 0,
\end{align*}
and the former expansion contains exactly one more occurrence of $11$. It follows that 
$$A_{\mathbf{r}}(n,d) \leq l+2-2^{k-1} \leq 2^k-1+2-2^{k-1} = 2^{k-1}+1. \qedhere$$
\end{proof}

\section{Some results for a general pattern} \label{sec:general_v}

In this section we investigate the behavior of the values $A_v(0,d)$ and $A_v(d)$ for a general pattern $v \in \{0,1\}^+$. Our first goal is to prove Theorem \ref{thm:general_bound}, which provides a universal upper bound on $A_v(0,d)$ and $A_v(d)$ when $|v| \geq 2$. In order to obtain this result we will need a technical lemma.

\begin{lem} \label{lem:prefix}
Let $x,y,v \in \{0,1\}^+$ be such that $|v| \geq 2$ and the initial digit of $x$ is different from the initial digit of $y$. Then there exists a word $w \in \{0,1\}^{|v|-1}$ such that $|wx|_{v} \not \equiv |wy|_{v} \pmod{2}$.
\end{lem}
\begin{proof}
If $|x|_{v} \not\equiv |y|_{v} \pmod{2}$, it is sufficient to take any $w$ such that $\Pref(v) \cap \Suf(w) = \varnothing$.
 For example, we can choose $w = 0^{|v|-1}$ if $v$ begins with $1$, and otherwise $w = 1^{|v|-1}$. Clearly, $|wx|_{v} = |x|_{v}$ and $|wy|_{v} = |y|_{v}$ so our assertion holds.

The case $|x|_{v} \equiv |y|_{v} \pmod{2}$ is a bit more complicated. Let $s$ be the longest proper suffix of $v$ which is simultaneously a prefix of one of $x,y$. Since $x,y$ begin with distinct digits, $s$ is nonempty and is a prefix of precisely one of these words. Without loss of generality assume that this is the case for $x$. Let $p$ be such that $v = ps$ and note that $|p| \leq |v|-1$.  We put $w = zp$ for some $z \in \{0,1\}^{|v|-1-|p|}$ (to be determined) and make the following claims, from which the result will follow:
\begin{itemize}
\item we can choose $z$ in such a way that $p$ is the longest word in $\Pref(v) \cap \Suf(w)$;
\item for such choice of $z$ we have $|wx|_{v} = |x|_{v} + 1$ and $|wy|_{v} = |y|_{v}$.
\end{itemize}

Let $P'$ denote the set of these $p' \in \Pref(v)$ such that $p$ is a proper suffix of $p'$. In order for the first claim to hold we need to ensure that $w=zp$ is simultaneously not in the form $z'p'$, where $p'\in P'$. This is vacuously true for any $z$ if $p'=\varnothing$. Otherwise, for each $p' \in P'$ there are precisely $2^{|v|-1-|p'|}$ ways to choose $z'$ (and thus also $z$) so that $p' \in \Pref(v) \cap \Suf(w)$. Hence, the number of ways to choose $z$ so that the first claim holds is
$$ 2^{|v|-1-|p|} - \sum_{p' \in P'} 2^{|v|-1-|p'|} \geq 2^{|v|-1-|p|} - \sum_{l=|p|+1}^{|v|-1} 2^{|v|-1-l}  \geq 1.$$ 

We move on to the second claim. Clearly, we have an occurrence of $v$ in $wx$ which does not already appear in $x$, and hence $|wx|_v \geq |x|_v +1  $. Conversely, for any such occurrence we have $v = p's'$ for some nonempty $p', s'$, where $p' \in \Suf(w)$ and $s' \in \Pref(x)$. We have $|p'| \leq |p|$ by the first claim and $|s'| \leq |s|$ by the initial assumption about $s$. Hence, $p'=p$, $s'=s$, which implies $|wx|_v = |x|_v + 1$. Along the same lines, if $wy$ were to contain an additional occurrence of $v$, then again $v = p's'$, where $p' \in \Suf(w)$ and $s' \in \Pref(y)$ are proper. Similarly, we get $|p'| \leq |p|$ so $|s'| \geq |s|$. But the last inequality is a contradiction due to the choice if $s$.
\end{proof}

\begin{proof}[Proof of Theorem \ref{thm:general_bound}]
It suffices to show that for any $d < 2^k$ and $n \in \N$ the postulated inequality holds with $A_v(d)$ replaced by $A_v(n,d)$. Write $d = 2^{\nu_2(d)} d'$ and let $z \in \{0,1\}^{\nu_2(d)}$ be such that $[z]_2 = n \bmod{2^{\nu_2(d)}}$ (we have $z = \epsilon$ if $\nu_2(d)=0$).

 We first deal with the special case when $v = 0^i$ for some $i \geq 2$ and $n < 2^{\nu_2(d)}$, i.e., $n=[z]_2$. Then we get
$$e_v(2^{i}d + n) = |(d')_2 0^i z|_v = 1+ |(d')_2 0^{i-1} z|_v = 1+ e_v(2^{i-1}d + n),   $$
which implies $A_v(n,d) \leq 2^{|v|}\leq 2^{k+ |v| - \nu_2(d)-1}$, as desired.
 
Hence, assume that $v$ contains a $1$ or $n \geq 2^{\nu_2(d)}$ (or both). Let $w \in \{0,1\}^{|v|-1}$ be a word obtained from Lemma \ref{lem:prefix} applied to $x = 0^{\ell(d')} z$ and $y = (d')_2 z$. We can find an integer $j \in \{0,1,\ldots, 2^{|v|+\ell(d')-1} - 1 \}$ such that $(n+jd)_2 = uwx$ for some $u \in \{0,1\}^*$. Consequently, we also have $(n+(j+1)d)_2 = uwy$. If $v$ begins with precisely $i \geq 0$ zeros and contains a $1$, by Lemma \ref{lem:prefix} we get the non-congruence
$$ e_{v}(n+(j+1)d) \equiv |0^i u w|_{v} + |wy|_v \not\equiv |0^i uw|_{v} + |wx|_v \equiv  e_{v}(n+jd) \pmod{2}.   $$
If $v=0^i$ for some $i \geq 2$, then the assumption $n \geq 2^{\nu_2(d)}$ guarantees that $uw$ contains a $1$, and we again obtain $e_{v}(n+(j+1)d)  \not \equiv e_{v}(n+jd) \pmod{2}$ by a similar computation (with the prefix $0^i$ omitted).
In either case, we get
$$A_v(n,d) \leq j + 1 \leq 2^{|v| +  \ell(d) - \nu_2(d)-1} \leq 2^{k+ |v| - \nu_2(d)-1},$$
and the result follows.
\end{proof}

As an immediate corollary of the theorem, we obtain that $A_v(d)$ (and consequently $A_v(0,d)$) is bounded by a linear function in $d$.

\begin{cor} \label{cor:general_bound}
For every pattern $v$ of length $|v| \geq 2$ and all integers $d \geq 1$ we have
$$ A_v(0,d) \leq  A_v(d) \leq 2^{|v|-\nu_2(d)} d.  $$
\end{cor}

The obtained inequalities are most likely far from optimal, which is discussed in more detail in the next section. 

Taking into account Theorem \ref{thm:Par17}, we have left to give an upper bound for $A_0(d)$. It turns out that we have the equality $A_{0}(d) = A_{\mathbf{t}}(d)$ for all $d \geq 1$, which immediately gives us a nice characterization by virtue of said theorem.

\begin{prop} \label{prop:A_0=A_t}
For all positive integers $d$ we have $A_{0}(d) = A_{\mathbf{t}}(d)$.
Therefore, for any positive integer $k$, we get
$$  \max_{1 \leq d \leq 2^k-2} A_{0}(d)  \leq  2^k$$
and
$$ \max_{1 \leq d < 2^k} A_0(d) = A_0(2^k-1) = \begin{cases}
2^k + 4, &\text{if } k \equiv 0 \pmod{2},\\
2^k &\text{otherwise}.
\end{cases}   $$
\end{prop}
\begin{proof}
Fix $d \geq 1$ and $n \in \N$. Observe that the value $A_0(n,d)$ is finite, since 
$$ e_{0}(2^{\ell(n)+1}d+n) = 1+ e_{0}(2^{\ell(n)}d+ n).$$
For any integer $l \geq \ell(d) + \ell(n+d A_{0}(n,d))$ we also get
$$ A_0(2^l +n,d) = A_0(n,d).  $$
At the same time, for each $i = 0,1, \ldots, A_{0}(d)-1$, we have the equality
$$ l+1 = \ell(2^l+n+id) = e_{1}(2^l+n+id) + e_{0}(2^l+n+id). $$
It follows that $e_{1}(2^l+n+id) \bmod{2}$ is constant with respect to $i$, and therefore 
$$A_{\mathbf{t}}(d) \geq A_{\mathbf{t}}(2^l+n,d) \geq A_0(2^l +n,d) = A_0(n,d).$$
Taking the supremum with respect to $n$, we obtain $A_{\mathbf{t}}(d) \geq A_0(d)$. The reverse inequality is obtained in a similar fashion.
\end{proof}

We move on to study the lengths of monochromatic progressions starting at $0$, namely the values $A_v(0,d)$. Observe that the case when $v$ is a block of zeros is rather uninteresting. Indeed, if $v = 0^i$ for some $i \geq 1$, then $e_v(2^i d) = 1 + e_v(2^{i-1} d)$ so $A_v(0,d) \leq 2^i$ for any $d \geq 1$. Therefore, we will focus on patterns $v$ belonging to the set
$$ V = \{ v \in \{0,1\}^*: |v| \geq 2 \text{ and } v \neq 0^i\}. $$
The following proposition exhibits infinite families of $d$ such that $A_v(0,d)$ is ``almost'' linear with respect to $d$. In particular, this means that for each fixed $v \in V$ the values $A_v(0,d)$ (and consequently $A_v(d)$) are arbitrarily large.

\begin{prop} \label{prop:limit_points}
Let $v  \in V$ and write $v = 0^i u 0^j$, where $i \geq 0, j \geq 0$, and $u$ begins and ends with a $1$. Then there exist $x,y \in \{0,1\}^{|v|-1}$ such that 
\begin{equation} \label{eq:limit_points_cong}
|0^i x|_{v} + |yx|_{v} + |y 0^j|_{v} \equiv 1 \pmod{2}.
\end{equation}
Furthermore, assume that $(x_{\min},y_{\min})$ is the lexicographically minimal solution to \eqref{eq:limit_points_cong} and put
\begin{align*}
 C_v &= \frac{[x_{\min}]_2}{2^{j+|v|-1}}, \\
 B_v &= [y_{\min}]_2 - 2^j C_v.
 \end{align*}
Fix $l,m \in \N$, where $m$ is odd and let $d_k = 2^{l+j}(2^k+1)m$ for $k \in \N$. Then we have
$$ A_v(0,d_k)  = \frac{C_v}{2^l m^2}d_k + O(1), $$
as $k \to \infty$, where the implied constant depends only on $v,l,m$.
In particular, if $m=1$, then for all $k \geq 2|v|-2$ we have
$$ A_v(0,d_k) = \frac{C_v}{2^l} d_k + B_v.$$
\end{prop}
\begin{proof}
For \eqref{eq:limit_points_cong} to hold it is necessary and sufficient that either precisely one or all three of the following congruences are satisfied:
\begin{align}
|y x|_{v} &\equiv 1 \pmod{2}, \label{eq:limit_points_cong1} \\
|0^i x|_{v} &\equiv 1 \pmod{2}, \label{eq:limit_points_cong2} \\
|y 0^j|_{v} &\equiv 1 \pmod{2}. \label{eq:limit_points_cong3}
\end{align}
If $v$ begins with a $0$, then for $x = 0^{i-1} u 0^j$ and $y=1^{|v|-1}$ out of the three above congruences only \eqref{eq:limit_points_cong2} holds. If $v$ ends with a $0$, then for $x = 1^{|v|-1}$ and $y=0^i u 0^{j-1}$ only \eqref{eq:limit_points_cong3} holds. Finally, if $v$ both begins and ends with a $1$, then  \eqref{eq:limit_points_cong2} and \eqref{eq:limit_points_cong3} can never be satisfied, and thus we can take $yx = v0^{|v|-2}$ so that \eqref{eq:limit_points_cong1} holds.

Fix $k$ such that $2^{k-2|v|+2} > m$. For each positive integer $n < 2^k/m$ let $w_n \in \{0,1\}^{k}$ satisfy $[w_n]_2 = mn$.  Write $w_n = x_n z_n y_n$, where $|x_n| = |y_n| = |v|-1$.
Then we have 
$$  e_{v}(d_kn) = |0^i w_n w_n  0^{l+j}|_{v} \equiv |0^i x_n|_{v} + |y_n x_n|_{v} + |y_n 0^j|_{v} \pmod{2} $$
so $e_{v}(d_kn) \equiv 1 \pmod{2}$ if and only if $(x,y) = (x_n,y_n)$ is a solution to \eqref{eq:limit_points_cong}. 

We are now going to show that for any $x,y \in \{0,1\}^{|v|-1}$ there exists $n$ such that $x_n=x$ and $y_n=y$. These conditions are equivalent to
\begin{equation} \label{eq:limit_points1}
2^{k -|v|+1}[x]_2 \leq  mn  < 2^{k -|v|+1}([x]_2+1)
\end{equation}
and
\begin{equation} \label{eq:limit_points2}
mn \equiv [y]_2 \pmod{2^{|v|-1}},
\end{equation}
respectively. Since
$$ 2^{k -|v|+1}([x]_2+1) - 2^{k -|v|+1}[x]_2 > 2^{|v|-1}m, $$
in each congruence class modulo $2^{|v|-1}$ there exists $n$ for which \eqref{eq:limit_points1} holds. Therefore, we obtain some $n$ simultaneously satisfying \eqref{eq:limit_points1} and \eqref{eq:limit_points2}. Moreover, we see that minimal such $n$ satisfies the stronger inequality
\begin{equation} \label{eq:limit_points3}
2^{k -|v|+1}[x]_2 \leq  mn  < 2^{k -|v|+1}[x]_2 + 2^{|v|-1}m.
\end{equation}

Now, if $n=A_v(0,d_k)$, then we must have $x_n=x_{\min}$ (though not necessarily $y_n = y_{\min}$). Plugging $x=x_{\min}$, $n = A_v(0,d_k)$, and $2^k = 2^{-(l+j)}d_k/m - 1$ into \eqref{eq:limit_points3}, after some manipulation we get the desired asymptotic expression for $A_v(0,d_k)$.

If $m=1$, we have $[w_n]_2 = n$, and thus the minimal $n$ satisfying  $e_{v}(d_kn) \equiv 1 \pmod{2}$ is  
\begin{align*}
A_v(0,d_k) &=n = [x_{\min} 0^{k-2|v|+2}y_{\min}]_2  = 2^{k -|v|+1}[x_{\min}]_2 + [y_{\min}]_2 \\
&= \frac{C_v}{2^l} d_k + [y_{\min}]_2 - 2^j C_v. \qedhere
\end{align*}
\end{proof}

Note that taking $v = 11$ and $l=1$ we obtain precisely case (ii) of Theorem \ref{thm:main1}. As an immediate corollary, we can determine an infinite set of limit points of the sequence $(A_v(0,d)/d)_{n \geq 0}$.

\begin{cor} \label{cor:limit_points}
For all $v \in V$ and $l,m \in \N, m \geq 1$ the number $C_v/(2^l m^2)$ is a limit point of the sequence $(A_v(0,d)/d)_{d \geq 1}$.
\end{cor}

Based on experimental calculations, we suspect that these are all nonzero limit points. In the next section we discuss this in more detail.

A typical distribution of the values $\log_2(A_{v}(0,d)/d)$ is shown in Figure \ref{fig:RS_limit_points1} below (with $v=11$).
\begin{figure}[h!]
    \includegraphics[width=\textwidth]{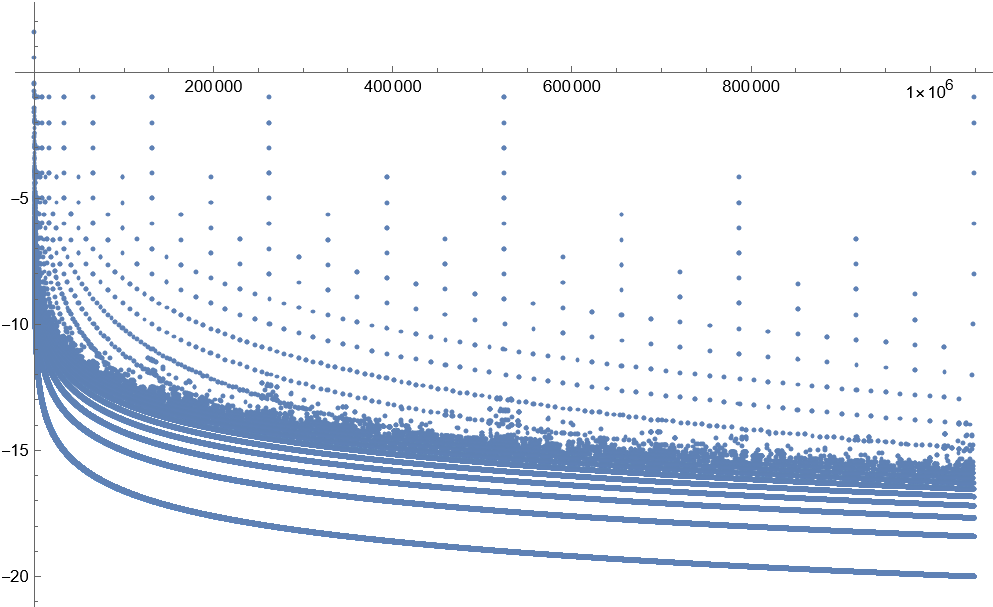}
    \caption{\label{fig:RS_limit_points1}The distribution of $\log_2(A_{\mathbf{r}}(0,d)/d)$}
\end{figure}

It is a matter of interest to determine the constants $C_v$ and (to a lesser extent) $B_v$ for $v \in V$.
From the definition we see that they are positive dyadic rational numbers, where additionally $C_v < 1$ (this inequality is strengthened in Proposition \ref{prop:Cv_inequality} below). Table \ref{tab:C_v} provides the values $C_v$ and $B_v$ for all $v \in V$ of length $|v| \leq 4$.

\begin{table}[H]  \renewcommand*{\arraystretch}{1.2}
\begin{tabular}{l|llllllllllllll}
$v$&$01$&$10$&$11$&$001$&$010$&$011$&$100$&$101$&$110$&$111$&$0001$&$0010$&$0011$&$0100$\\  \hline 
$C_v$ &$\frac{1}{2}$&$\frac{1}{4}$&$\frac{1}{2}$&$\frac{1}{4}$&$\frac{1}{4}$&$\frac{1}{2}$&$\frac{1}{16}$&$\frac{1}{4}$&$\frac{1}{4}$&$\frac{1}{2}$&$\frac{1}{8}$&$\frac{1}{8}$&$\frac{3}{8}$&$\frac{1}{16}$  \\

$B_v$ &$\frac{1}{2}$&$\frac{1}{2}$&$\frac{1}{2}$&$\frac{3}{4}$&$\frac{5}{2}$&$\frac{1}{2}$&$\frac{3}{4}$&$\frac{3}{4}$&$\frac{1}{2}$&$\frac{5}{2}$&$\frac{7}{8}$&$\frac{3}{4}$&$\frac{5}{8}$&$\frac{3}{4}$  \\

\end{tabular}
\vspace{\baselineskip}

\begin{tabular}{l|lllllllllll}
$v$&$0101$&$0110$&$0111$&$1000$&$1001$&$1010$&$1011$&$1100$&$1101$&$1110$&$1111$  \\ \hline
$C_v$ &$\frac{1}{4}$&$\frac{1}{4}$&$\frac{1}{2}$&$\frac{1}{64}$&$\frac{1}{8}$&$\frac{1}{8}$&$\frac{3}{8}$&$\frac{1}{16}$&$\frac{1}{4}$&$\frac{1}{4}$&$\frac{1}{2}$ \\

$B_v$ &$\frac{3}{4}$&$\frac{1}{2}$&$\frac{5}{2}$&$\frac{7}{8}$&$\frac{7}{8}$&$\frac{3}{4}$&$\frac{5}{8}$&$\frac{11}{4}$&$\frac{11}{4}$&$\frac{5}{2}$&$\frac{13}{2}$ \\
\end{tabular}
\caption{The values of $C_v$ and $B_v$ for $v \in V$ such that $|v| \leq 4$}
\label{tab:C_v}
\end{table}

In the case when $v$ ends with a $1$ the constants $C_v$ and $B_v$ can be quickly calculated without directly solving the congruence \eqref{eq:limit_points_cong}, as the following proposition shows. An explicit description should also be possible to obtain when $v$ ends with a $0$, although additional complications arise in this case.

\begin{prop} \label{prop:limit_points2}
Let $v \in V$ and assume that $v$ ends with a $1$. Write $v = ps$, where $s$ is the lexicographically minimal proper and nonempty suffix of $v$. Then the lexicographically minimal solution to \eqref{eq:limit_points_cong} is given by 
\begin{align*}
x_{\min} &= s0^{|p|-1}, \\
 y_{\min} &= \begin{cases}
0^{|s|-1} p &\text{if } p \neq 0, \\
0^{|v|-2} 1 &\text{if } p = 0.
\end{cases}
\end{align*}
In particular, we have $C_v = 2^{-|s|} [s]_2$ and $B_v = \max \{[p]_2,1\} - C_v$.
\end{prop}
\begin{proof}
Because $v$ ends with a $1$, the congruence \eqref{eq:limit_points_cong3} has no solution $y \in \{0,1\}^{|v|-1}$. Our task reduces to showing that $(x_{\min},y_{\min})$ as in the statement is the lexicographically minimal pair such that precisely one of \eqref{eq:limit_points_cong1} and \eqref{eq:limit_points_cong2} holds.

We first show that the pair  
$$(x_1,y_1) =  (s 0^{|p|-1},0^{|s|-1} p)$$
is the minimal solution to \eqref{eq:limit_points_cong1}. To see this, assume that $|yx|_{v} \equiv 1 \pmod{2}$ for some $x,y \in \{0,1\}^{|v|-1}$ such that $(x,y) \leq (x_1,y_1)$ lexicographically. This implies that $yx$ contains at least one occurrence of $v$, hence we can write $v = p's'$ with $p',s'$ both nonempty, where $x = s' q$ and $y = r p'$ for some $q,r \in \{0,1\}^*$. We also have $s' q \leq s 0^{|p|-1}$. Since $s'$ ends with a $1$, this implies $|s'| \leq |s|$. Hence, we also have $s' \leq s$, which yields $s'=s$ by the choice of $s$, and consequently $x=x_1$. This also means that $p'=p$, and thus $y=y_1$, since $y_1$ is the lexicographically minimal word of length $|v|-1$ with suffix $p$.
 
If $v$ begins with a $1$ (which falls under the case $p \neq 0$), then $(x_1,y_1)$ is already the minimal solution to \eqref{eq:limit_points_cong} as \eqref{eq:limit_points_cong2} is never satisfied. Hence, assume from now on that $v$ begins with $i \geq 1$ zeros. Then the pair
$$(x_2, y_2) = (0^{i-1} u , 0^{|v|-1})$$
is the minimal one satisfying \eqref{eq:limit_points_cong2}. 

We always have $(x_1, y_1) \leq (x_2, y_2)$ in the lexicographical sense with equality if and only if $p$ is a block of zeros. By minimality of $s$ we must have  $s = 0^{i-1} u$ in such a case so the condition $(x_1,y_1) = (x_2,y_2)$ is equivalent to $p = 0$.
In the case $(x_1, y_1) < (x_2 , y_2)$  the pair $(x_1,y_1)$ remains the minimal solution to \eqref{eq:limit_points_cong}. Otherwise, if $(x_1, y_1) = (x_2 , y_2)$, then $(x,y)=(s0^{|p|-1},0^{|v|-2} 1) = (0^{i-1}u,0^{|v|-2} 1)$ is the successor of $(x_1,y_1)$. Observe that such pair $(x,y)$ still satisfies \eqref{eq:limit_points_cong2}. On the other hand, if \eqref{eq:limit_points_cong1} held for such $(x,y)$ then we would have an occurrence of $v$ in the word $yx = 0^{|v|-2} 1 s 0^{|p|-1}$. Since $v \neq 1 s$ and $v$ ends with a $1$, there would exist a representation $v = p's'$, where $s'$ is a proper prefix of $s$ but this contradicts the minimality of $s$. 

It remains to observe that $2^{-|s|}[s]_2 = 2^{-|x_{\min}|} [x]_2 = C_v$.
\end{proof}

In the following proposition we provide a good upper bound on $C_v$ for all $v \in V$.

\begin{prop} \label{prop:Cv_inequality}
Let $v \in V$ and write $v = w0^j$, where $w$ ends with a $1$. Then we have
$$C_v \leq \begin{cases} 
 1/2 &\text{if } j=0, \\
2^{-2j} &\text{if } j \geq 1.
\end{cases}
$$
\end{prop}
\begin{proof}
First, assume that $v$ ends with a $1$. If $v = 01^l$ or $v=1^l$ for some $l \geq 1$, then with the notation of Proposition \ref{prop:limit_points2} we have $s=1$, which gives $C_v = 1/2$. Otherwise, $s$ must begin with a $0$ so $C_v < 1/2$.

Now, assume that $v$ ends with a $0$, that is, $j \geq 1$. Write $w = 0^i u$, where $u$ begins with a $1$. We are going to show that the congruence \eqref{eq:limit_points_cong} has a solution with $x = 0^{j-1} 1  0^{|v|-j-1}$, which in turn gives $[x_{\min}]_2 \leq [x]_2 = 2^{|v|-j-1}$, and therefore $C_v \leq  2^{-2j}$. We consider a few possibilities.  

First, if $0^ix$ contains $v$, then we must have $u=1$ and $i \geq 1$, which gives $|0^ix|_v=1$. In such a case, for $y=1^{|v|-1}$ the pair $(x,y)$ is a solution to \eqref{eq:limit_points_cong}. 

Hence, assume that $|0^ix|_v = 0$. If $0^{j-1} 1 0^j$  is not a suffix of $v$, put $y= 0^{j-1}w$. Then we get $|y0^j|_v = 1$ and $|yx|_v = 0$ so $(x,y)$ satisfies \eqref{eq:limit_points_cong}. Otherwise, if $0^{j-1} 1 0^j$  is a suffix of $v$ (which entails $2j+1 \leq |v|$), write $v =  q 0^{j-1} 1 0^j$ and put $y = 0^{2j-1}q$. Then we get $|y0^j|_v = 0$ and $|yx|_v = |0^{2j-1} v 0^{|v|-2j-1} |_v = 1$, and thus \eqref{eq:limit_points_cong} again holds.
\end{proof}

One can quickly check that the obtained bound is sharp. The cases where $C_v = 1/2$ were described in the proof. At the same time, we get $C_v = 2^{-2j}$ for $v=10^j$, since the lexicographically minimal solution to \eqref{eq:limit_points_cong} is $(x_{\min},y_{\min}) = (0^{j-1}1,0^j)$.

\section{Further problems and conjectures} \label{sec:problems}

In this section we suggest possible improvements to the presented results and pose related problems and conjectures. They are mostly based on the analysis of the values $A_v(0,d)$ for $|v| \leq 4, d \leq 2^{20}+1$ and $A_v(d)$ for $|v| \leq 3, d \leq 2^{12}+1$. We note that precise numerical results concerning $A_v(d)$ are computationally expensive to obtain even for small $|v|$. In the final section we discuss how the values $A_{v}(d)$ can be calculated. 
The numerical data is available at the repository \cite{Git}.

To begin, for $v \in \{0,1\}^+$ let us consider the upper limits
\begin{align*}
L_{0,v} &= \limsup_{d \to \infty} \frac{A_v(0,d)}{d}, \\
L_v &= \limsup_{d \to \infty} \frac{A_v(d)}{d}
\end{align*}
Also let $C_v$ be as in Proposition \ref{prob:limit_points} and define $C_1 = 1, C_{0^i} = 0$. From the results in this paper we can conclude that 
 $$ C_v \leq L_{0,v} \leq L_v \leq 2^{|v|}.$$
 In particular, $L_{0,v}, L_v$ are finite and nonzero (with the exception $L_{0,0^i} = 0$).
Therefore, it is natural to ask for the precise values $L_{0,v}, L_v$, and more generally the entire sets of limit points.

\begin{prob} \label{prob:limit_points}
For $v\in\{0,1\}^+$ compute $L_{0,v}$ and $L_v$. More generally, compute all limit points of each of the sequences $(A_v(0,d)/d)_{d \geq 1}$ and $(A_v(d)/d)_{d \geq 1}$.
\end{prob}

A closely related problem is the following.

\begin{prob} \label{prob:upper_bound}
Do there exist constants $M_{0,v}$ and $M_v$ such that for all sufficiently large $d$ we have
\begin{align*}
A_{0,v}(d) &\leq L_{0,v} d + M_{0,v}, \\
A_{v}(d) &\leq L_{v} d + M_{v},
\end{align*}
where equality occurs for infinitely many $d$?
\end{prob}

Let us briefly summarize what is known so far for $A_v(0,d)$. We have $L_{0,1} = 1$ and $L_{0,11}=1/2$ by Theorems \ref{thm:MSS11} and \ref{thm:main1}, respectively. These results also give a linear upper bound on $A_v(0,d)$ for $v=1$ and $v=11$, which becomes an equality infinitely often. Additionally, Corollary \ref{cor:limit_points} exhibits infinitely many points of the sequence $(A_v(0,d)/d)_{d \geq 1}$ for $v$ in the set $V = \{v \in \{0,1\}^*: |v| \geq 2 \text{ and } v \neq 0^i\}$. Based on numerical evidence, we strongly believe that these are all nonzero limit points. Thus, we pose the following conjecture, where we also include the case $v=1$ with $C_1 = 1$.

\begin{con} \label{con:limit_points}
For all $v \in V \cup \{1\}$ the set of limit points of $(A_v(0,d)/d)_{d \geq 1}$ is equal to
$$ \left\{ \frac{C_v}{2^l m^2}: l,m \in \N, m \geq 1 \right\} \cup \{0\}.$$
In particular, 
$$ L_{0,v} = C_v. $$
\end{con}

Additionally, according to the computations we believe that the choice of $d =d_k = 2^j(2^k+1)$ (as in Proposition \ref{prop:limit_points}) yields the maximal values of $A_v(0,d)$ in the sense of the following conjecture. This time the patterns $v \in \{1,11\}$ need to be excluded due to the special case $d = 2^k-1$, where the postulated inequality does not work. 

\begin{con} \label{con:upper_bound}
Let $v \in V \setminus \{11\}$ and write $v = u0^j$, where $u$ ends with a $1$. Then $M_{0,v} = B_v$ and for all $d \geq 2^{2|v|+1}$ we have
$$A_v(0,d) \leq C_v d + B_v,$$
where equality occurs if and only if $d = 2^j(2^k+1)$ for some $k \geq 2|v|-2$.
\end{con}

Some observations concerning Problems \ref{prob:limit_points} and \ref{prob:upper_bound} can also be made for $A_v(d)$. Experimental computations suggest that if $v$ ends with a $1$, then the upper limits $L_v$ and $L_{\overline{v}}$ are identical and equal to $C_v$:
\begin{equation} \label{eq:limsup}
L_v = L_{\overline{v}} = C_v.
\end{equation}
This would immediately imply that for $v$ ending with a $1$ we have$L_{0,v} = L_v$, which means that ``large'' values $A_v(d)$ are approximated well by $A_v(0,d)$. Comparing the results stated in the introduction, this is indeed the case for the words $v =1,v=11$. Going one step further, one may also expect that the limit points of the sequences $(A_v(d)/d)_{d \geq 1}$ and $(A_v(0,d)/d)_{d \geq 1}$ are identical, i.e., Conjecture \ref{con:limit_points} also applies to $A_v(d)$.

Furthermore, it seems that the answer to the question concerning the existence of $M_v$ in Problem \ref{prob:upper_bound} is also affirmative. The equality $A_v(d) = L_v d + M_v$ apparently holds for sufficiently large $d=2^k-1$ when $v \in \{1,11,0,00\}$, or $d=2^k+1$ otherwise. In the case of the Thue--Morse and Rudin--Shapiro sequences we state the following conjectures. 

\begin{con} \label{con:TM_bound}
For all $d \geq 1$ we have
$$ A_{\mathbf{t}}(d) \leq d+5. $$
Moreover, equality occurs if and only if $d=2^k-1$ for some even $k \geq 2$.
\end{con}

\begin{con} \label{con:RS_bound}
For all $d \geq 23$ we have
$$ A_{\mathbf{r}}(d) \leq \frac{1}{2}(d+7). $$
Moreover, equality occurs if and only if $d=2^k-1$ for some odd $k \geq 5$.
\end{con}

Problems \ref{prob:limit_points} and \ref{prob:upper_bound} seem quite hard to solve even for fixed $v$.
We believe that it is easier to find an upper bound for $A_v(d)$ in the flavor of Theorem \ref{thm:Par17}. 

\begin{prob}
Compute $\max_{1 \leq d < 2^k} A_v(d)$ and determine $d$ for which this maximum is reached.
\end{prob}

For example, in the case of the Rudin--Shapiro sequence Theorem \ref{thm:main2} together with Conjecture \ref{con:RS_bound} (if valid) would imply the following. 

\begin{con}
For all integers $k \geq 5$ we have
$$  \max_{1 \leq d \leq 2^k-2} A_{\mathbf{r}}(d)  \leq  2^{k-1}$$
and
$$   A_{\mathbf{r}}(2^k-1) =  \begin{cases}
2^{k-1} + 1, &\text{if } k \equiv 0 \pmod{2},\\
2^{k-1} + 3 &\text{otherwise}.
\end{cases} $$
\end{con}

We move on to some miscellaneous problems. The first one concerns a possible generalization of the relation $A_0(d) = A_{\mathbf{t}}(d)$ and is related to the supposed equality \eqref{eq:limsup}.
\begin{prob}
Describe the set of $d$ such that $A_v(d) = A_{\overline{v}}(d)$.
\end{prob}
According to our computations (for $|v| \leq 3$), when $v \not\in \{00,11,000,111\}$ the equality $A_v(d) = A_{\overline{v}}(d)$ apparently holds for all $d \geq 1$. One might suspect that this pattern continues, namely $A_v(d) = A_{\overline{v}}(d)$ for all $d$ when $v$ is not of the form $0^i$ or $1^i$.   

Next, recall the equality $A_{\mathbf{t}}(2d) = A_{\mathbf{t}}(d)$, which implies $A_0(2d) = A_0(d)$ by Proposition \ref{prop:A_0=A_t}. It turns out that a similar relation is true for the Rudin--Shapiro sequence.
\begin{prop} \label{prop:d_odd}
For all $d \geq 1$ we have $A_{\mathbf{r}}(2d) = A_{\mathbf{r}}(d)$.
\end{prop}
\begin{proof}
It is sufficient to show the for any fixed odd $d$ and $j \geq 1$ we have $A_{\mathbf{r}}(2^j d) = A_{\mathbf{r}}(d)$. First, for any $n,i \in \N$ we have $r_{n+id} = r_{2^jn + i \cdot 2^jd}$, which yields $A_{\mathbf{r}}(2^j d) \geq A_{\mathbf{r}}(d)$.

Conversely, take any $n \in \N$ and write $n = 2^j m + l$, where $m,l \in \N$ and $l < 2^j$. Then for any $i \in \N$ we have
$$ r_{n+i\cdot 2^jd} \equiv \begin{cases}
r_l + r_{m+id}  \pmod{2}  &\text{if } l < 2^{j-1}, \\
r_l + s_{m+id} \pmod{2}  &\text{if } l \geq 2^{j-1}
\end{cases}    $$
(recall that $\mathbf{s} =  (r_{2n+1})_{n \geq 0}$). Therefore,
$$  A_{\mathbf{r}}(n,2^jd) \leq \max\{A_{\mathbf{r}}(m,d), A_{\mathbf{s}}(m,d)\} \leq \max\{A_{\mathbf{r}}(d), A_{\mathbf{s}}(d)\} = A_{\mathbf{r}}(d),$$
where we used $A_{\mathbf{r}}(d) =  A_{\mathbf{s}}(d)$, a consequence of Lemma \ref{lem:reversed_subwords}. The result follows by taking the supremum over $n$.
\end{proof}

It seems interesting to ask whether a similar relation holds for any $v$.
\begin{prob} \label{prob:odd_d}
Describe the set of $d$ such that $A_v(2d) = A_v(d)$.
\end{prob}

Based on the calculations, $v=01,10$ are the only other patterns of length at most $3$ such that the equality holds for all $d$. Nevertheless, in all cases we have $A_v(2d) \geq A_v(d)$ with equality for most $d$.

The same question can be asked for the values $A_v(0,d)$.
\begin{prob} 
For $v$ ending with a $0$ describe the set of $d$ such that $A_v(0,2d) = A_v(0,d)$.
\end{prob}
We exclude $v$ ending with a $1$ as then $A_v(0,2d) = A_v(0,d)$ follows immediately from the fact that $e_v(2n) =  e_v(n)$. The case when $v$ ends with a $0$ seems to be nontrivial and the considered equality holds for most but not all $d$.

Most of the results in this paper concern the growth of $A_v(0,d)$ and $A_v(d)$ with particular attention paid to specific large terms. However, it is also interesting to look at the distribution of all the values $A_v(0,d)$ and $A_v(d)$.

\begin{prob}
Describe the distribution of the values $A_v(0,d)$ and $A_v(d)$.
\end{prob}

Here it is easy to see that the similarity in behavior of large terms of the sequences $(A_v(0,d))_{d \geq 1}$ and $(A_v(d))_{d \geq 1}$ does not carry over to the whole sequences. In Figures \ref{fig:RS_distribution1} and \ref{fig:RS_distribution2} below we provide histograms for the example sequences $(A_{\mathbf{r}}(0,d))_{d \geq 1}$ and $(A_{\mathbf{r}}(d))_{d \geq 1}$. For better readability, the (rare) large values $A_v(0,d)$ and $A_v(d)$ are omitted.

\begin{figure}[h!]  
  \includegraphics[width=0.75\textwidth]{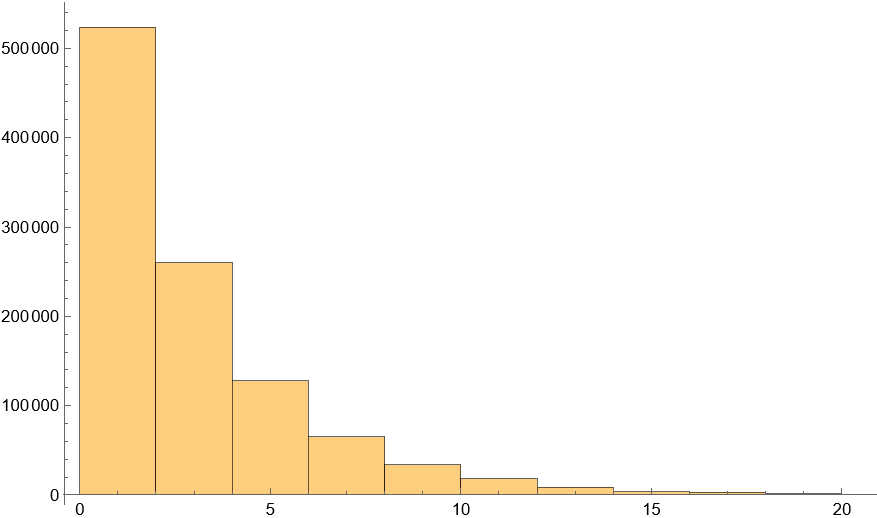}
  \caption{\label{fig:RS_distribution1}The distribution of $A_{\mathbf{r}}(0,d)$}
\end{figure}

\begin{figure}[h!]
    \includegraphics[width=0.75\textwidth]{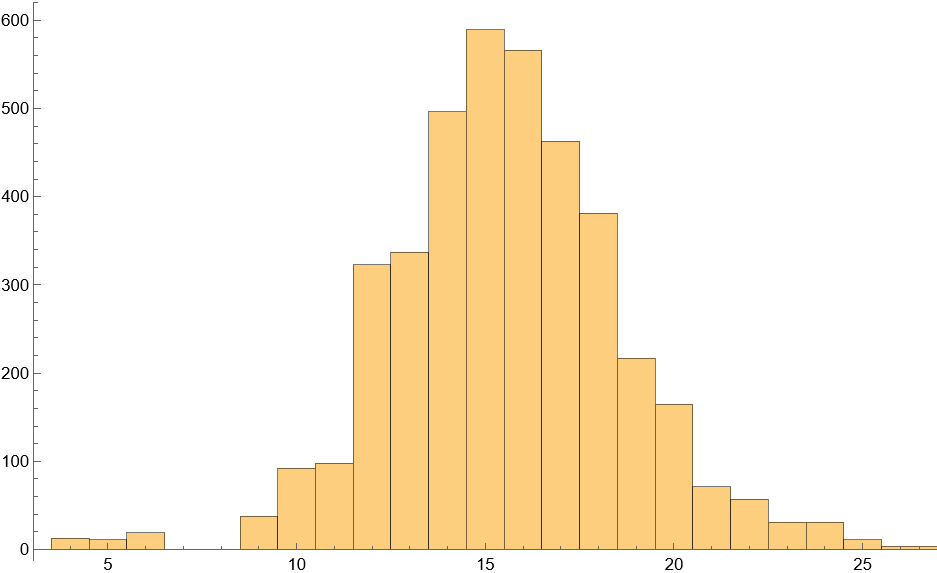}
    \caption{\label{fig:RS_distribution2}The distribution of $A_{\mathbf{r}}(d)$}
\end{figure}

The distribution of the values $A_{\mathbf{r}}(0,d)$ resembles the geometric distribution with success probability $1/2$ (after proper scaling). This phenomenon may be explained by the following simple heuristic. For roughly half of the values $d$ we have $r_d =1$, and thus $A_{\mathbf{r}}(0,d)=1$. Assuming independence of $r_d$ and $r_{3d}$, for the remaining values $d$ around half of the time we should have $r_{3d} =1$, implying $A_{\mathbf{r}}(0,d)=3$, and so on. On the other hand, the distribution of $A_{\mathbf{r}}(d)$ seems approximately normal.

A related problem is the following.

\begin{prob}
Determine the values attained by the sequences $(A_v(0,d))_{d \geq 1}$ and $(A_v(d))_{d \geq 1}$.
\end{prob}

It appears that $A_v(0,d)$ takes on all positive odd values and finitely many even values (none if $v$ ends with a $1$). At the same time, $A_v(d)$ attains almost all integer values in the interval $[10,30]$ for each considered $v$. We expect that the set of all values of $A_v(d)$ may have asymptotic density $1$.

Since the sequences $(A_v(0,d))_{d \geq 1}$ and $(A_v(d))_{d \geq 1}$ describe binary expansions of integers one may wonder whether they are $2$-regular (in the sense of Allouche and Shallit \cite{AS92}). We expect that the general answer is negative.

\begin{prob}
Are the sequences $(A_v(0,d))_{d \geq 1}$ and $(A_v(d))_{d \geq 1}$ $2$-regular? 
\end{prob}

Finally, we can generalize our setting and study the values $e_{b,v}(n)$ modulo $m$, where $e_{b,v}(n)$ counts the occurrences of a pattern $v$ in the base-$b$ expansion of $n$, and $m \geq 2$ is an integer. It is interesting to ask whether similar results hold for monochromatic arithmetic progressions in $m$-colorings of $\N$ induced by these sequences. 
\begin{prob}
Generalize the study to other bases and moduli.
\end{prob}

\section{The calculation of $A_v(d)$}  \label{sec:calculation}

In this section we discuss how the values $A_v(d)$ can be calculated for a general $v$. Let us denote $g_v(n) = e_v(n) \bmod{2}$ and $\mathbf{g}_v = (g_v(n))_{n \geq 0}$. Our approach is an extension of the method used to calculate particular values of $A_{\mathbf{r}}(d)$ in the proofs of Theorems \ref{thm:main2} and \ref{thm:main3}. There are certain differences depending on whether or not $v$ contains a $1$. We now describe the case when $v$ contains a $1$ and at the end of the section highlight the modifications that need to be made when $v=0^i$.

The first step is to determine a substitution such that $\mathbf{g}_v$ is its fixed point. To this end, write $\mathbf{g}_v$ as a concatenation of $2^{|v|-1}$-aligned blocks:
$$ \mathbf{g}_v = G_v(0) G_v(1) G_v(2) \cdots,  $$
where $G_v(n) = g_v(2^{|v|-1}n) \cdots g_v(2^{|v|-1}(n+1)-1)$ for each $n \in \N$. Also let 
$$ \Sigma_v = \{G_v(n): n \in \N   \}. $$
In this way we can identify $\mathbf{g}_v$ with the infinite word $\mathbf{G}_v = (G_v(n))_{n \geq 0}$ over the alphabet $\Sigma_v$.  In the following proposition we give the desired description.

\begin{prop} \label{prop:g_v_substitution}
For any $v\in\{0,1\}^*$ which contains a $1$ we have the following.
\begin{enumerate}[label={\textup{(\roman*)}}]
\item $\Sigma_v = \{G_v(l), \overline{G_v(l)}: 0 \leq l < 2^{|v|-1}\}$.
\item The word $\mathbf{g}_v$ is a fixed point of a substitution $\gamma_v \colon \Sigma_v^* \to \Sigma_v^*$ defined by $$\gamma_v(G_v(n)) = G_v(2n) G_v(2n+1).$$
\item $\gamma_v(\overline{G}) = \overline{\gamma_v(G)}$ for any $G \in \Sigma_v$.
\end{enumerate}
\end{prop}
\begin{proof}
We start with (i). Choose $n \in \N$ and write $n = 2^{|v|-1} m + l$, where $0 \leq l < 2^{|v|-1}$. Then for each $j = 0,1,\ldots,2^{|v|-1}-1$ we have
\begin{equation} \label{eq:g_v_congruence}
g_v(2^{|v|-1}n + j) \equiv g_v(n) + g_v(2^{|v|-1}l + j) \pmod{2}.
\end{equation}
Therefore, either $G_v(n) = G_v(l)$ or $G_v(n) = \overline{G_v(l)}$. 

The other inclusion follows immediately for the blocks $G_l$. To show that $\overline{G_v(l)} \in \Sigma_v$, it is sufficient to take any $n \in \N$ such that $n \equiv l \pmod{2^{|v|-1}}$ and $g_v(n)=1$, and use \eqref{eq:g_v_congruence}. A suitable choice is for example $n = 2^{2|v|-1} [v]_2+l = [v0^{|v|}(l)_2]_2$ if $v$ begins with a $1$, and $n = 2^{2|v|-1} [v]_2+ 2^{|v|-1}(2^{|v|}-1)+l=[v1^{|v|}(l)_2]_2$ otherwise.

In order to prove (ii) note that for all $j =0,1,\ldots, 2^{|v|-1}-1$ we have
\begin{align}
g_v(2^{|v|}n + 2j) &\equiv  g_v(2^{|v|-1}n + j) + g_v(2j) \pmod{2}, \label{eq:g_v_congruence2} \\
g_v(2^{|v|}n + 2j+1) &\equiv  g_v(2^{|v|-1}n + j) + g_v(2j+1) \pmod{2}.  \label{eq:g_v_congruence3}
\end{align}
If we assume that $G_v(n) = G_v(m)$, then the right-hand side of \eqref{eq:g_v_congruence2} and \eqref{eq:g_v_congruence3} remains the same after replacing $n$ by $m$, which implies $G_v(2n) G_v(2n+1) = G_v(2m) G_v(2m+1)$.

Part (iii) follows directly from the congruences   \eqref{eq:g_v_congruence2} and \eqref{eq:g_v_congruence3}.
\end{proof}

Having this result, the substitution $\gamma_v$ can be effectively obtained by generating sufficiently many terms of $\mathbf{g}_v$.

Now, a monochromatic arithmetic progression in $\mathbf{g}_v$ of length $l$ and difference $d$ is contained in a subword of length $(l-1)d+1$. At the same time, for $|v| \geq 2$ and all $d \geq 1$ we have the upper bound
\begin{equation} \label{eq:a_priori_bound}
A_v(d) \leq  2^{\ell(d)+|v|-\nu_2(d)-1},
\end{equation}
which follows from Theorem \ref{thm:general_bound}. The same inequality also holds when $v=1,0$ and $d \neq 2^k-1$ by Theorem \ref{thm:Par17} and the relations $A_{\mathbf{t}}(2d) = A_{\mathbf{t}}(d), A_0(d) = A_{\mathbf{t}}(d)$. Since the values $A_{\mathbf{t}}(2^k-1)$ are already known, we omit this case in the further considerations. As a result, the progression of maximal length has to be contained in a subword of $\mathbf{g}_v$ of length 
$$(2^{\ell(d)+|v|-\nu_2(d)-1}-1)d+1 < 2^{2\ell(d)+|v|-\nu_2(d)-1}.$$
Further, any such subword is contained in a word of the form $\gamma_v^{s}(G_n G_{n+1})$ for some $n \in \N$, where $s = 2 \ell(d) - \nu_2(d)$.
We state this observation as a proposition. Here and in the sequel $\Sub  _2(w)$ denotes the set of length $2$ subwords of a word $w$.
\begin{prop} \label{prop:progressions_in_subwords}
Let $v \in \{0,1\}^*$ contain a $1$. Then for each $d \geq 1$ (except for $d=2^k-1$ in the case $v=1$) the value $A_v(d)$ is equal to the maximal length of a monochromatic arithmetic progression of difference $d$ in the words $\gamma_v^s(GH)$, where $GH \in \Sub_2(\mathbf{G}_v)$ and $s = 2\ell(d)-\nu_2(d)$.
\end{prop}
This essentially reduces our task to finding all subwords of length $2$ of a fixed point of a $2$-uniform morphism. Recall that a substitution $\varphi \colon \Sigma^* \to \Sigma^*$ is a $2$-uniform morphism if $|\varphi(a)| = 2$ for all $a \in \Sigma$. The following proposition is rather standard, however we include a proof for completeness.

\begin{prop}
Let $\Sigma$ be a finite alphabet. Let $\mathbf{a} = (a_{n})_{n \geq 0}$ be a fixed point of a $2$-uniform morphism $\varphi \colon \Sigma^* \to \Sigma^*$. For $k \geq 1$ define
$$  S_k = \Sub_2(\varphi^k(a_0))$$ 
and assume that $S_{K+1} = S_{K}$ for some $K$. Then we have
$$\Sub_2(\mathbf{a})= S_K.$$
\end{prop}
\begin{proof}
Since the words $\varphi^k(a_0)$ tend to $\mathbf{a}$, we have
$$  \Sub_2(\mathbf{a}) = \bigcup_{k=1}^{\infty} S_k.  $$
Since also $S_k \subset S_{k+1}$, it is sufficient to prove that $S_{K+n+1} = S_{K+n}$ for all $n \in \N$. The case $n=1$ is our assumption. Now, suppose that the claim holds for some $n$ and choose any word $w \in S_{K+n+2}$. Then $w$ is a subword of $\varphi(u)$ for some $u \in S_{K+n+1}$. But by the inductive assumption we have $u \in S_{K+n}$ and thus also $  w \in S_{K+n+1}$. The result follows.
\end{proof}

Since there are only finitely many distinct words of length $2$ we eventually get $S_{k+1} = S_k$. 
In our case this means that the subwords of $\mathbf{G}_v$ of length $2$ can be effectively computed by considering $\gamma_v^k(G_0)$ for subsequent $k$.

The approach presented above can be improved in a few ways from the computational point of view. Firstly, the amount calculations required to compute $A_v(d)$ according to the method outlined above depends on the length of the binary expansion of $d$, regardless of the actual value $A_v(d)$. However, based on smaller-scale calculations, one can observe that most values $A_v(d)$ are small. Therefore, it is beneficial to first filter out $d$ yielding these values. In general, assume that $f \colon \N_+ \to \R_{+}$ is some function and we would like to compute the values $A_v(d)$ satisfying the inequality $A_v(d) \leq f(d)$. By the same calculation as before, a monochromatic progression of difference $d$ and length greater than $f(d)$ would have to appear in some block $\gamma_v^s(GH)$, where $GH \in \Sub_2(\mathbf{G}_v)$ and $s = \lceil \log_2 (f(d)d) \rceil - |v| + 1$. If for given $d$ this is not the case, then we obtain the value $A_v(d)$. This process can be repeated for the remaining $d$ using another function $f_1$, and so on. As the final bound one can use the general inequality \eqref{eq:a_priori_bound}. In our calculations we have used $f(d) = 20, f_1(d)=35, f_2(d) = d$ (or $f_2(d)=d+5$ for the Thue--Morse sequence) with considerable improvement in calculation time. 

Secondly, parts (i) and (iii) of Proposition \ref{prop:g_v_substitution} imply that $GH \in \Sub_2(\mathbf{G}_v)$ if and only if $\overline{GH} \in \Sub_2(\mathbf{G}_v)$. Since negation only affects the color but not the length of monochromatic arithmetic progressions, it is sufficient to search only one of $\gamma_v^s(GH), \gamma_v^s(\overline{GH})$ for each $GH \in \Sub_2(\mathbf{G}_v)$. This essentially cuts the task in half.

Finally, in the case of the Thue--Morse sequence we can restrict our attention to $d$ odd by the equality $A_{\mathbf{t}}(2d) = A_{\mathbf{t}}(d)$. The same reduction can be done for the Rudin--Shapiro sequence by Proposition \ref{prop:d_odd}, and possibly other sequences $\mathbf{g}_v$ (see Problem \ref{prob:odd_d}).
 
To conclude this section, we discuss the (slight) modifications that need to be made when $v=0^i$ for some $i \geq 1$. In particular, we see that Proposition \ref{prop:g_v_substitution}(ii) does not hold for $v=0^i$, since for example $G_v(0) = G_v(1) = 0^{2^{i-1}}$ but $G_v(2 \cdot 0) = 0^{2^{i-1}} \neq 10^{2^{i-1}-1} = G_v(2\cdot 1)$. In order to deal with this, we define $\gamma_v$ on the set $\Sigma_v$ with an additional element $X$, which is later mapped to the block $G_v(0)$.

\begin{prop} \label{prop:g_v_substitution2}
If $v = 0^i$ for some $i \geq 1$, then we have the following.
\begin{enumerate}[label={\textup{(\roman*)}}]
\item $\Sigma_v = \{G_v(l), \overline{G_v(l)}: 0 \leq l \leq 2^{|v|-1}\}$.
\item Define the substitution $\gamma_v \colon (\Sigma_v \cup \{X\})^* \to (\Sigma_v \cup \{X\})^*$ by
\begin{align*}
\gamma_v(X)  &= X G_v(1), \\
\gamma_v(G_v(n)) &= G_v(2n) G_v(2n+1), \quad n \geq 1.
\end{align*}
The word $\mathbf{g}_v$ is the image of the fixed point of $\gamma_v$ starting with $X$ under the map $X \mapsto G_v(0)$ and $G \mapsto G$ for $G \in \Sigma_v$.
\item $\gamma_v(\overline{G}) = \overline{\gamma_v(G)}$ for any $G \in \Sigma_v$
\end{enumerate}
\end{prop}
\begin{proof}
The proof is similar as in Proposition \ref{prop:d_odd} so we primarily highlight the differences. In part (i) we again take $n \in \N$ and write $n = 2^{|v|-1} m + l$, where $0 \leq l < 2^{|v|-1}$. If $l \neq 0$, then the congruence \eqref{eq:g_v_congruence} holds and $G_v(n)$ is equal to one of $G_v(l), \overline{G_v(l)}$. If $l = 0$ and $m \neq 0$, then we instead get for $j=0,1,\ldots 2^{|v|}$ the congruence 
$$ g_v(2^{|v|-1}n+j) \equiv g_v(n) + g_v(2^{2|v|-2} + j). $$
It follows that $G_v(n)$ is equal to one of $G_v(2^{|v|-1}), \overline{G_v(2^{|v|-1})}$. In the remaining case $n=0$ the block $G_v(0)$ trivially belongs to the set on the right-hand side.

Conversely, we need to show that the blocks $\overline{G_v(l)}$ belong to $\Sigma_v$. If $0 < l < 2^{|v|-1}$, then we have $g_v(n)=1$ for $n =  2^{2|v|} + 2^{|v|-1} +l = [1v1(l)_2]_2$ and by \eqref{eq:g_v_congruence} we get  $\overline{G_v(l)} = G_v(n) \in \Sigma_v$. One can also check that 
$\overline{G_v(0)} = 1^{2^{|v|}} = G_v(2^{|v|+1}+1)$ and $\overline{G_v(2^{|v|-1})} =  G_v(2^{|v|})$, thus both blocks also belong to $\Sigma_v$.

In part (ii), if $n \geq 1$ the congruence \eqref{eq:g_v_congruence2} does not hold for $j=0$, and we get two cases:
 $$g_v(2^{|v|}n+2j) \equiv 
\begin{cases} 
g_v(2^{|v|-1}n) +1  \pmod{2} &\text{if } j=0, \\
g_v(2^{|v|-1}n+j) \pmod{2} &\text{if } 1 \leq j < 2^{|v|-1}.
\end{cases} $$
At the same time, the congruence \eqref{eq:g_v_congruence3} remains identical. In any case, we deduce that $\gamma_v(G)$ is well-defined for $G \in \Sigma_v$.  It is also clear that $\mathbf{g}_v$ is the image of the fixed point of $\gamma_v$ under the map given in the statement.
\end{proof}

The remaining part of the procedure for $v=0^{i}$ stays mostly the same as before. The only difference is that after finding all length $2$ subwords $GH$ in the fixed point of $\gamma_v$ and calculating $\gamma_v^s(GH)$ for suitable $s$, one needs to apply the map $X \mapsto G_v(0) = 0^{2^i}$ to the word $\gamma_v^s(X G_v(1))$. The improvements to the method discussed earlier apply as well. We also note that the case $v=0$ does not need to be considered separately due to the equality $A_{0}(d) = A_{\mathbf{t}}(d)$ of Proposition \ref{prop:A_0=A_t}.

\section{Description of the supplemental files}  \label{sec:supplemental}
The following supplemental files to this paper are available in the repository \cite{Git}.
\begin{enumerate}
\item \texttt{longest\_progressions.csv} \\
The file contains the values $A_v(d)$ for all patterns $v$ such that $|v| \leq 3$ and all differences $d \in [1,2^{12}+1]$. The first row consists of headings of columns (corresponding to $v$), while the $(d+1)$th row contains the values $A_v(d)$.
\item \texttt{longest\_progressions\_from\_0.csv} \\
The file contains the values $A_v(0,d)$ for all patterns $v$ such that $|v| \leq 4, v \neq 0^i$ and all differences $d \in [1,2^{20}+1]$. The first row consists of headings of columns (corresponding to $v$), while the $(d+1)$th row contains the values $A_v(0,d)$.
\item \texttt{section\_3\_algorithm.nb} \\
The file contains an implementation of the algorithm described in Section \ref{sec:main1}. The algorithm is then applied to the discussed cases, yielding results stated in Proposition \ref{prop:application_1} and \ref{prop:application_2}.
\end{enumerate}

\section*{Acknowledgements}
The research is supported by the grant of the National Science Centre (NCN), Poland, no.\ UMO-2020/37/N/ST1/02655.

\bibliographystyle{amsplain}
\bibliography{references}

\providecommand{\bysame}{\leavevmode\hbox to3em{\hrulefill}\thinspace}
\providecommand{\MR}{\relax\ifhmode\unskip\space\fi MR }
\providecommand{\MRhref}[2]{%
  \href{http://www.ams.org/mathscinet-getitem?mr=#1}{#2}
}
\providecommand{\href}[2]{#2}
\begin{thebibliography}{10}

\bibitem{AGMNP23}
Ibai Aedo, Uwe Grimm, Neil Mañibo, Yasushi Nagai, and Petra Staynova,
  \emph{Monochromatic arithmetic progressions in automatic sequences with group
  structure}, 2023.

\bibitem{AGNS22}
Ibai Aedo, Uwe Grimm, Yasushi Nagai, and Petra Staynova, \emph{Monochromatic
  arithmetic progressions in binary {T}hue-{M}orse-like words}, Theoret.
  Comput. Sci. \textbf{934} (2022), 65--80. \MR{4491756}

\bibitem{AS92}
Jean-Paul Allouche and Jeffrey Shallit, \emph{The ring of {$k$}-regular
  sequences}, Theoret. Comput. Sci. \textbf{98} (1992), no.~2, 163--197.

\bibitem{AS99}
\bysame, \emph{The ubiquitous {P}rouhet-{T}hue-{M}orse sequence}, Sequences and
  their applications ({S}ingapore, 1998), Springer Ser. Discrete Math. Theor.
  Comput. Sci., Springer, London, 1999, pp.~1--16.

\bibitem{AS03}
\bysame, \emph{Automatic sequences: Theory, applications, generalizations},
  Cambridge University Press, Cambridge, 2003.

\bibitem{ACF06}
Sergey~V. Avgustinovich, Julien Cassaigne, and Anna~E. Frid, \emph{Sequences of
  low arithmetical complexity}, Theor. Inform. Appl. \textbf{40} (2006), no.~4,
  569--582. \MR{2277050}

\bibitem{Gel67}
A.~O. Gelfond, \emph{Sur les nombres qui ont des propri\'{e}t\'{e}s additives
  et multiplicatives donn\'{e}es}, Acta Arith. \textbf{13} (1967/68), 259--265.
  \MR{220693}

\bibitem{Mat}
Wolfram~Research{,} Inc., \emph{Mathematica, {V}ersion 13.0.0}, Champaign, IL,
  2021.

\bibitem{MSS11}
Johannes~F. Morgenbesser, Jeffrey Shallit, and Thomas Stoll, \emph{Thue-{M}orse
  at multiples of an integer}, J. Number Theory \textbf{131} (2011), no.~8,
  1498--1512.

\bibitem{Par17}
Olga~G. Parshina, \emph{On arithmetic index in the generalized {T}hue-{M}orse
  word}, Combinatorics on words, Lecture Notes in Comput. Sci., vol. 10432,
  Springer, Cham, 2017, pp.~121--131.

\bibitem{Git}
Bartosz Sobolewski, \emph{Monochromatic arithmetic progressions},
  \url{https://github.com/BartoszSobolewski/Monochromatic-arithmetic-progressions},
  2022.

\bibitem{vdW27}
Bartel~L. van~der Waerden, \emph{Beweis einer baudetschen vermutung}, Nieuw
  Arch. Wisk. \textbf{15} (1927), 212--216.

\end{thebibliography}

\end{document}